\documentclass[11pt,reqno]{amsart}
\usepackage{enumerate}
\usepackage{url}
\usepackage{cite}
\usepackage{comment}
\usepackage{graphicx}
\usepackage{amsfonts, amsmath, amssymb, amscd, amsthm, bm, cancel}
\usepackage[
linktocpage=true,colorlinks,citecolor=magenta,linkcolor=blue,urlcolor=magenta]{hyperref}
\usepackage{tikz}
\usetikzlibrary{arrows,shapes,trees,backgrounds}
\usepackage{colortbl}
\usepackage{xcolor}
\usepackage[margin=1.1in]{geometry}
\newtheorem{lemma}{Lemma}[section]

\newtheorem{proposition}[lemma]{Proposition}
\newtheorem{theorem}[lemma]{Theorem}
\newtheorem{definition}[lemma]{Definition}

\theoremstyle{definition}
\newtheorem*{ack}{Acknowledgments}
\allowdisplaybreaks
\numberwithin{equation}{section}
\begin{document}
\title{Anisotropic Gauss curvature flow of complete non-compact graphs}
\author[S. Pan]{Shujing Pan}
\author[Y. Wei]{Yong Wei}
\address{School of Mathematical Sciences, University of Science and Technology of China, Hefei 230026, P.R. China}
\email{\href{mailto:psj@ustc.edu.cn}{psj@ustc.edu.cn}}
\email{\href{mailto:yongwei@ustc.edu.cn}{yongwei@ustc.edu.cn}}
\begin{abstract}
In this paper, we consider the anisotropic $\alpha$-Gauss curvature flow for complete noncompact convex hypersurfaces in the Euclidean space with the anisotropy determined by a smooth closed uniformly convex Wulff shape. We show that for all positive power $\alpha>0$, if the initial hypersurface is complete noncompact and locally uniformly convex, then the solution of the flow exists for all positive time.
\end{abstract}

\maketitle
\tableofcontents

\section{Introduction}\label{set0}
Let $f\in C^\infty(\mathbb{S}^n)$ be a smooth positive function on the unit sphere $\mathbb{S}^n$ such that $\bar{\nabla}^2f+fg_{\mathbb{S}^n}$ is uniformly positive definite, where $\bar{\nabla}^2f$ denotes the Hessian of $f$ with respect to the round metric $g_{\mathbb{S}^n}$ on $\mathbb{S}^n$. Then $f$ determines a unique smooth uniformly convex hypersurface $\mathcal{W}$ with $f$ being its support function. We call $\mathcal{W}$ the Wulff shape associated with the anisotropy $f\in C^\infty(\mathbb{S}^n)$. When $f\equiv 1$ is a constant, the Wulff shape $\mathcal{W}$ reduces to the unit round sphere $\mathbb{S}^n$ and this is the isotropic case. In this paper, we consider the following anisotropic $\alpha$-Gauss curvature flow  $F:M^n\times[0,T)\to\mathbb{R}^{n+1}$
\begin{equation}\label{flow0}
 \frac{\partial}{\partial t}F(p,t)=-f(\nu)K^{\alpha}(p,t)\nu(p,t),\qquad \alpha>0,
\end{equation}
in the Euclidean space, where $K(p,t)$ is the Gauss curvature of $\Sigma_t=F(M^n,t)$ at
 $F(p,t)$, $\nu(p,t)$ is the unit normal vector at $F(p,t)$ pointing outward of the convex hull of $\Sigma_t$. The anisotropic function $f(\nu)$ is evaluated at the unit normal $\nu$.

In the isotropic case with $f\equiv 1$, the flow \eqref{flow0} is known as the $\alpha$-Gauss curvature flow for a given power $\alpha>0$. The case $\alpha=1$ corresponds to the classical Gauss curvature flow which was proposed by Firey \cite{firey1974shapes} to model the erosion of strictly convex stones as they tumble on a beach. To obtain the model, he assumed that the stones were of uniform density, that their wear was isotropic and that the number of collisions in a region was proportional
to the set of normal directions of the region. Then the rate of wear is
proportional to the Gauss curvature and the process is just the Gauss curvature flow, i.e., the flow \eqref{flow0} with $\alpha=1$ and $f\equiv 1$. For any smooth closed and uniformly convex hypersurface in $\mathbb{R}^{n+1}$, Tso \cite{Tso85} proved that there exists a unique smooth solution to the Gauss curvature flow on a finite maximal time interval $0\leq t<T$ and the solution remains uniformly convex and contracts to a point as $t\to T$. This result was extended to $\alpha$-Gauss curvature flow for all $\alpha>0$ by Chow \cite{chow1985deforming}. The asymptotic behavior of the solution as $t\to T$ was studied later in many works. In particular, Andrews \cite{andrews1999gauss} proved that any smooth closed and uniformly convex surface in $\mathbb{R}^3$ contracts along the Gauss curvature flow to a round point as $t\to T$. The case
 $\alpha={1}/{(n+2)}$ corresponds to the affine normal flow and  Andrews \cite{andrews1996contraction} proved that the solution converges to a normalized ellipsoid after reparametrizing and rescaling about the final point. The  superaffine case $\alpha>1/(n+2)$ of the flow \eqref{flow0} was also widely studied and the asymptotical shape is always spherical in this case, see \cite{andrews2000motion,andrews2016flow,chow1985deforming,brendle2017asymptotic,guan2017entropy}. We refer the readers to \cite[Chapters 15-17]{andrews2020extrinsic} for an excellent exposition of the smooth closed solutions of the $\alpha$-Gauss curvature flow.

Allowing some anisotropy in the material of the stone, the wearing process of the tumbling stones is modeled as the anisotropic Gauss curvature flow \eqref{flow0} for some positive function $f$ on the sphere $\mathbb{S}^n$. Andrews \cite{andrews2000motion} studied the anisotropic $\alpha$-Gauss curvature flow \eqref{flow0} incorporated with a general anisotropy $f$ and showed that the flow \eqref{flow0} with $\alpha>0$ contracts any smooth closed convex hypersurface to a point in finite time. Moreover, for $\alpha\in({1}/{(n+2)},{1}/{n}]$ he proved that the solution converges to a self-similar shrinking solution after rescaling.  In his proof, the condition $\alpha\leq{1}/{n}$ is necessary to get the lower bound of the Gauss curvature. Recently, curvature estimate of the anisotropic Gauss curvature flow for $\alpha\in (0,\frac{1}{n-1}]$ was studied by Choi, Kim and Lee \cite{choi2023LpMinkowski}. Weak convergence in certain sense of the flow \eqref{flow0} was studied by B\"{o}r\"{o}czky and Guan \cite{BG24} and was applied to study the weak solution of $L_p$ Minkowski problem. The smooth convergence of a smooth closed convex solution of the anisotropic Gauss curvature flow \eqref{flow0} for $\alpha>{1}/{n}$ is not yet known and is still open.

 In this paper, we care about the complete non-compact case. We first review some results in the literature on this aspect. Ecker and Huisken \cite{ecker1989mean,ecker1991interior} studied the evolution of entire graphs by mean curvature flow, and proved that if the initial hypersurface is a graph of locally Lipschitz continuous function on $\mathbb{R}^n$, then the solution exists for all time. The mean curvature flow for complete graph over a bounded domain was studied by S\'{a}ez and Schn\"{u}rer \cite{SS14}. Recently, the evolution of complete non-compact graphs by fully nonlinear curvature flow has been considered. Choi and Daskalopoulos \cite{choi2016q_k} extended the techniques of Ecker and Huisken to the $\mathcal{Q}_k$ flow and proved the long-time existence under the weak convexity assumption. Choi, Daskalopoulos, Kim, and Lee \cite{choi2019evolution} extended the method to the $\alpha$-Gauss curvature flow (i.e., the flow \eqref{flow0} with $f\equiv 1$). For all $\alpha>0$, if the initial surface is a complete noncompact and locally uniformly convex hypersurface embedded in $\mathbb{R}^{n+1}$,
  they proved the global existence of the solution. The asymptotical behavior of the flow for $\alpha>{1}/{2}$ was studied later by Choi, Choi and Daskalopoulos \cite{choi2022convergence}. Alessandroni and Sinestrari \cite{AS15} studied
 the evolution of complete non-compact convex graphs by general symmetric curvature function $F$ under some conditions on $F$. The evolution of complete noncompact hypersurfaces by inverse mean curvature flow was also studied by Daskalopoulos and Huisken \cite{DH22} and Choi and Daskalopoulos \cite{CD-IMCF}.

 A natural question is whether the techniques can be extended to complete noncompact hypersurfaces evolving by the anisotropic $\alpha$-Gauss curvature flow (\ref{flow0}) for general anisotropy $f\in C^{\infty}(\mathbb{S}^n)$.  In this work, we consider the case $f$ being the support function of a smooth closed, uniformly convex  Wulff shape $\mathcal{W}$ in $\mathbb{R}^{n+1}$. We shall show the global existence of the flow \eqref{flow0} for complete non-compact and locally uniformly convex hypersurfaces.  In order to state our result, we recall the following definition of locally uniformly convex hypersurfaces as in \cite{choi2019evolution}.
 \begin{definition}
    We denote by $C^2_{\mathcal{H}}(\mathbb{R}^{n+1})$ the class of $C^2$ complete hypersurfaces embedded in $\mathbb{R}^{n+1}$.
    For a convex hypersurface $M\in C^2_{\mathcal{H}}(\mathbb{R}^{n+1})$, we denote by $\lambda_{\min}^M(p)$ the smallest principal
    curvature of $M$ at $p\in M$. Given any complete convex hypersurface $\Sigma$, which may not be necessarily of $C^2$, and $p\in\Sigma$, we define $\lambda_{\min}^{\Sigma}(p)$
    by
    \begin{equation*}
  \lambda_{\min}^{\Sigma}(p):=\sup\{\lambda_{\min}^M(p):p\in M,~ M\in C^2_{\mathcal{H}}(\mathbb{R}^{n+1}), \Sigma\subset \text{the convex hull of}\ M\}
    \end{equation*}
    and we say that (i). $\Sigma$ is strictly convex, if $\lambda_{\min}^{\Sigma}(p)>0$ for all points $p\in\Sigma$; (ii). $\Sigma$ is uniformly convex, if there is a constant $m>0$ satisfying $\lambda_{\min}^{\Sigma}(p)\geq m>0$ for all $p\in\Sigma$; and (iii). $\Sigma$ is locally uniformly convex, if for any compact subset $A\subset\mathbb{R}^{n+1}$, there is a constant $m_{A}>0$ satisfying $\lambda_{\min}^{\Sigma}(p)\geq m_A$ for all $p\in\Sigma\cap A$.
  \end{definition}
Our main result is stated as follows:
 \begin{theorem}\label{main thm}
  Assume that $\mathcal{W}$ is a smooth closed and uniformly convex Wulff shape which encloses a convex body $\widehat{\mathcal{W}}$. Let $f\in C^{\infty}(\mathbb{S}^n)$ be the support function of $\mathcal{W}$ with respect to the point $z_0\in \widehat{\mathcal{W}}$ satisfying the property in Proposition \ref{shift point}.
    Assume that $\Sigma_0$ is a complete, non-compact, locally uniformly convex hypersurface embedded in $\mathbb{R}^{n+1}$ given by an immersion
    $F_0: M^n\to\mathbb{R}^{n+1}$ with $\Sigma_0=F_0(M^n)$. Then for any $\alpha>0$, there exists a complete, non-compact, smooth and strictly
    convex solution $\Sigma_t:=F(M^n,t)$ of the flow
     \begin{equation}\label{flow}
 \left\{\begin{aligned}
   &\frac{\partial}{\partial t}F(p,t)=-f(\nu)K^{\alpha}(p,t)\nu(p,t),\qquad \alpha>0,\\
   &\lim_{t\to0}F(p,t)=F_0(p,t),
   \end{aligned}
   \right.
 \end{equation}
     which is defined for all time $0<t<+\infty$.
 \end{theorem}

For the proof, we assume that the initial hypersurface $\Sigma_0$ is expressed as a complete convex graph over a domain $\Omega\subset \mathbb{R}^n$, that is, $\Sigma_0=\{(x,u_0(x)):~x\in\Omega\}$ for some function $u_0:\Omega\to \mathbb{R}$, where $\Omega$ may be bounded or unbounded. This is ensured by the theorem of H. Wu \cite{wu1974spherical}:
 \begin{theorem}[\cite{wu1974spherical}]\label{cpt sup}
   	Let $\Sigma$ be a complete, non-compact, locally uniformly convex hypersurface embedded in $\mathbb{R}^{n+1}$. Then there exists a function $u_0:\Omega\to\mathbb{R}$ defined on a convex open domain $\Omega\subset\mathbb{R}^n$ such that $\Sigma=\{(x,u_0(x)):~x\in\Omega\}$ and
    \begin{itemize}
    	\item [(i)] $u_0$ attains its minimum in $\Omega$ and $\inf_{\Omega}u_0>0$;
    	\item [(ii)] if $\Omega\neq\mathbb{R}^{n}$, then $\lim_{x\to x_0}u_0(x)=+\infty$ for all $x_0\in\partial\Omega$;
    	\item [(iii)] if $\Omega$ is unbounded, then $\lim_{R\to+\infty}(\inf_{\Omega\setminus B_R(0)}u_0)=+\infty$.
    \end{itemize}  	
   \end{theorem}
\noindent Then we will work on the graphical solution of the flow. As in the work \cite{choi2019evolution} on the isotropic case, there are three ingredients in the proof of Theorem \ref{main thm}:  a local bound from above on the gradient function of $\Sigma_t$; a local bound from below on the smallest principle curvature $\lambda_{\min}$ of $\Sigma_t$; and a local bound from above on the speed $f(\nu)K^{\alpha}$ of $\Sigma_t$. We shall first assume the initial data is smooth and strictly convex in order to apply the known parabolic PDE theory, and the above local a priori estimates are proved for smooth solutions. An approximation argument implies a smooth solution for locally uniformly convex hypersurfaces.

   In our anisotropic case, a crucial new observation is that  by choosing a suitable point $z_0\in\widehat{\mathcal{W}}$ (see Proposition \ref{shift point}) the support function $f$ of the Wulff shape $\mathcal{W}$ with respect to $z_0$ satisfies the property  $\langle f(x)x+\bar{\nabla} f(x),e_{n+1}\rangle\leq 0$ whenever $\langle x,e_{n+1}\rangle\leq 0$, where $e_{n+1}$ is a fixed direction. This is important for us to obtain the local a priori estimates using the maximum principle.  By Theorem \ref{cpt sup}, a locally uniformly convex hypersurface $\Sigma$ can be written as a graph over a convex domain $\Omega\subset\mathbb{R}^n$. Let $e_{n+1}$ be the direction such that
\begin{equation}
  \langle\nu(p), e_{n+1}\rangle\leq 0,\quad \forall p\in \Sigma,
\end{equation}
  where $\nu(p)$ is the unit normal vector of $\Sigma$ at the point $p$ pointing outward of the convex hull of $\Sigma$. Then Proposition \ref{shift point} guarantees that the anisotropic unit normal vector field $\nu_f:=f(\mathbf{\nu})\mathbf{\nu}+\bar{\nabla} f(\mathbf{\nu})$  of $\Sigma$ satisfies
   \begin{equation}
     \langle \nu_f(p),e_{n+1}\rangle \leq 0,\quad \forall p\in \Sigma
   \end{equation}
as well.  Based on this, we modify the flow \eqref{flow} by changing the direction of the evolution from the unit normal $\nu$ to the anisotropic unit normal $\nu_f$:
 \begin{equation}\label{s1.flow2}
\frac{\partial}{\partial t}F(p,t)=-K^{\alpha}(p,t)\nu_f(p,t),\qquad \alpha>0,
 \end{equation}
This modified flow \eqref{s1.flow2} is equivalent to the original flow \eqref{flow} as their speeds only differ by a tangent vector field $\bar{\nabla}f(\nu)$ and the corresponding solutions are the same up to a time-dependent tangential diffeomorphism. The advantage of this modification is that we can cancel some bad terms in the evolution equations in order to apply the maximum principle.


The paper is organized as follows: In \S \ref{sec1}, we collect and prove some properties on the Wulff shape, and the anisotropic $\alpha$-Gauss curvature flow \eqref{s1.flow2}. In \S \ref{sec.evl} we derive the basic evolution equations along the flow \eqref{s1.flow2}. In  \S \ref{sec2} - \S \ref{sec4} we prove the crucial local a priori estimates on gradient function, principal curvatures and the speed function. In \S \ref{sec5} we complete the proof of Theorem \ref{main thm}.

\begin{ack}
The authors would like to thank the referees for helpful suggestions.  The research was surpported by National Key R and D Program of China 2021YFA1001800 and 2020YFA0713100, China Postdoctoral Science Foundation No.2022M723057. S. Pan would like to thank Professor Jiayu Li for his constant support.
\end{ack}

\section{The Wulff shape and the modified flow}\label{sec1}

In this section, we first review some properties of the Wulff shape and prove the existence of an approximate inner point inside the Wulff shape $\mathcal{W}$ such that the property in Proposition \ref{shift point} holds. Then we consider the anisotropy $f$ as the support function of the Wulff shape $\mathcal{W}$ with respect to this point and modify the flow \eqref{flow} to the equivalent flow \eqref{s1.flow2} by adding a tangential term.
\subsection{The Wulff shape}
Let $\mathcal{W}\subset\mathbb{R}^{n+1}$ be a smooth closed and uniformly convex Wulff shape and $\widehat{\mathcal{W}}$ be the convex body bounded by $\mathcal{W}$. The Gauss map of $\mathcal{W}$ is defined as the unit outward normal $G=\nu:\mathcal{W}\to \mathbb{S}^n$ which is a non-degenerate diffeomorphism. Then $\mathcal{W}$ can be parametrized by the inverse Gauss map $G^{-1}: \mathbb{S}^n\to\mathcal{W}$. Fix any point $z\in\widehat{\mathcal{W}}$, the support function $f_z:\mathbb{S}^n\to\mathbb{R}_+$ of $\mathcal{W}$ with respect to $z$
is defined by
$$f_z(x)=\sup\{\langle x,y-z\rangle:y\in\mathcal{W}\}.$$
Since $\mathcal{W}$ is uniformly convex, the supremum is attained at a point $y$ such that $x$ is the outer normal of $\mathcal{W}$ at $y$, and
\[y-z=f_z(x)x+\bar{\nabla} f_z(x),\]
where $\bar{\nabla}$ is the covariant derivative with respect to the round metric $g_{\mathbb{S}^n}=(\bar{g}_{ij})$ of $\mathbb{S}^n$. This gives the embedding $G^{-1}: \mathbb{S}^n\to \mathbb{R}^{n+1}$ with image equals $\mathcal{W}$. The Weingarten map of $\mathcal{W}$ can also be recovered by the support function (see e.g. \cite{andrews2000motion,urbas1991expansion}). Let
\begin{equation}\label{wulff eq-1}
	\mathfrak{r}_{ij}=\bar{\nabla}^2_{ij}f_z+f_z\bar{g}_{ij},
\end{equation}
where   $\bar{\nabla}^2_{ij}f$ denotes the Hessian of $f$ with respect to the $\bar{g}_{ij}$. The eigenvalues of $\mathfrak{r}_{ij}$ with respect to $\bar{g}_{ij}$ are the principal radii of curvature $\mathfrak{r}_i=1/{\lambda_i}, i=1,\cdots,n$. As $\mathcal{W}$ is smooth, closed and uniformly convex, the matrix $\mathfrak{r}_{ij}$  is positive definite at any point and there are uniformly positive lower and upper bounds of the eigenvalues.

Suppose $e_{n+1}$ is a fixed direction in $\mathbb{R}^{n+1}$.  We next show that there exists a point $z_0$ contained in  $\widehat{\mathcal{W}}$, such that the support function $f_{z_0}$ of $\mathcal{W}$ with respect to $z_0$ satisfies $\langle {f}_{z_0}(x)x+\bar{\nabla} {f}_{z_0}(x),e_{n+1}\rangle\leq 0$ whenever $\langle x,e_{n+1}\rangle\leq 0$ for $x\in \mathbb{S}^n$. First we have the following simple observation.
   \begin{lemma}\label{convex hypersurface}
  	Suppose that $\mathcal{W}\subset\mathbb{R}^{n+1}$ is a smooth closed and uniformly convex hypersurface and $X$ denotes its position vector field.
  	If $\langle X, e_{n+1}\rangle$ attains its maximum at some point $X_0\in \mathcal{W}$, then the outer normal vector at
  	$X_0$ is $e_{n+1}$.
  \end{lemma}
  \begin{proof}
  	Assume that $\{e_i, i=1,\cdots,n\}$ are orthonormal frame of the tangent space of $\mathcal{W}$, $\nu$ is the outer normal vector field of $\mathcal{W}$.
  	By the maximum principle, at $X_0$, we have
  	\begin{equation}\label{eq-2}
  		\nabla_{e_i}\langle X, e_{n+1}\rangle=\langle e_i, e_{n+1}\rangle=0,\ \forall i=1,\cdots,n,
  	\end{equation}
  	and
  	\begin{eqnarray}\label{eq-3}
  		\nabla^2\langle X, e_{n+1}\rangle(X_0)\leq 0.
  	\end{eqnarray}
  	Here $\nabla$ denotes the Levi-Civita connection with respect to the induced metric on $\mathcal{W}$. Equation \eqref{eq-2} implies that $e_{n+1}$ is a normal vector field of $\mathcal{W}$ at $X_0$. On the other hand, by the Gauss formula $\nabla_i\nabla_jX=-h_{ij}\nu$,
  	and $\mathcal{W}$ is uniformly convex, \eqref{eq-3} leads to $\langle \nu, e_{n+1}\rangle\geq 0$ at $X_0$. Therefore,
  	$e_{n+1}=\nu$ at $X_0$.
  \end{proof}
  \begin{proposition}\label{shift point}
  	Let $\mathcal{W}\subset\mathbb{R}^{n+1}$ be a smooth closed and uniformly convex Wulff shape and $\widehat{\mathcal{W}}$ be the convex body bounded by $\mathcal{W}$. Then there exists an interior point $z_0\in\widehat{\mathcal{W}}$, such that the support function $f_{z_0}$ of $\mathcal{W}$ with respect to the point $z_0$ satisfies
  \begin{equation*}
    \langle f_{z_0}(x)x+\bar{\nabla }f_{z_0}(x),e_{n+1}\rangle\leq 0
  \end{equation*}
  whenever $\langle x,e_{n+1}\rangle\leq 0$, where $x\in \mathbb{S}^n$.
  \end{proposition}
  \begin{proof}
  	 For fixed direction $e_{n+1}$, there
  	is a great circle $\mathbb{S}^1$ on $\mathbb{S}^n$ whose position vector is  perpendicular to $e_{n+1}$. Then,  $\mathbb{S}^1$ divides  $\mathbb{S}^n$ into two parts
  	denoted by
  	$$\mathbb{S}^+:=\{x\in \mathbb{S}^n:\langle x, e_{n+1}\rangle\geq 0\},$$
  	$$\mathbb{S}^-:=\{x\in \mathbb{S}^n:\langle x, e_{n+1}\rangle\leq 0\}.$$
  	Since the Gauss map $G$ is a non-degenerate diffeomorphism between $\mathcal{W}$ and $\mathbb{S}^n$, the inverse of $G$ maps $\mathbb{S}^1$ to a closed curve $\varGamma$ on $\mathcal{W}$.

  Let $f$ be the support function of $\mathcal{W}$ with respect to the original point $o\in \mathbb{R}^{n+1}$. Then $x\in \mathbb{S}^n$ is the outer normal vector at the point $y=f(x)x+\bar{\nabla} f(x)$ on $\mathcal{W}$. Therefore, on $\varGamma$ the outer
  	normal vector $x$ is perpendicular to $e_{n+1}$. Meanwhile, $\varGamma$ divides
  	$\mathcal{W}$ into two parts denoted by
  	$$\varGamma^+:=\{y=f(x)x+\bar{\nabla} f(x)\in\mathcal{W}:\langle x,e_{n+1}\rangle\geq 0\},$$
  	$$\varGamma^-:=\{y=f(x)x+\bar{\nabla} f(x)\in\mathcal{W}:\langle x,e_{n+1}\rangle\leq 0\}.$$
  	On the one hand, we assume that
  	$$\max_{x\in \mathbb{S}^-}\langle f(x)x+\bar{\nabla} f(x),e_{n+1}\rangle=\langle f(x_0)x_0+\nabla f(x_0),e_{n+1}\rangle=\langle y_0,e_{n+1}\rangle$$
  	where $x_0\in \mathbb{S}^-$, $y_0=f(x_0)x_0+\bar{\nabla} f(x_0)\in\varGamma^-$. On the other hand, we assume that
  	$$\max_{x\in \mathbb{S}^n}\langle f(x)x+\bar{\nabla} f(x),e_{n+1}\rangle=\langle f(x_1)x_1+\bar{\nabla} f(x_1), e_{n+1}\rangle=\langle y_1,e_{n+1}\rangle,$$
  	where $x_1\in \mathbb{S}^n$, $y_1=f(x_1)x_1+\bar{\nabla} f(x_1)\in \mathcal{W}$.
  	By Lemma \ref{convex hypersurface}, we know that $x_1=e_{n+1}$, that means $x_1\in \mathbb{S}^+\setminus \mathbb{S}^1$ and
  	$y_1\in\varGamma^+\setminus\varGamma$.
  	
  	Since $\mathcal{W}$ is uniformly convex, the point $ty_1+(1-t)y_0$ is
  	contained in $\widehat{\mathcal{W}}$ for any $t\in(0,1)$. Moreover,
  	\begin{eqnarray*}
  		\langle y_0-(ty_1+(1-t)y_0), e_{n+1}\rangle&=&t\langle y_0-y_1,e_{n+1}\rangle\leq 0.
  	\end{eqnarray*}
  	Therefore, for any fixed $t_0\in(0,1)$, let $z_0=t_0y_1+(1-t_0)y_0\in\widehat{\mathcal{W}}$, then
  	\begin{eqnarray}\label{s2.1}
  		\langle y-z_0, e_{n+1}\rangle\leq \langle y_0-z_0, e_{n+1}\rangle\leq 0, \ \forall y\in\varGamma^-.
  	\end{eqnarray}

  Fix the point $z_0\in\widehat{\mathcal{W}}$, we note that the support function $f_{z_0}(x)$ of the Wulff shape
   $\mathcal{W}$ with respect to $z_0$ and the support function $f(x)$ of the Wulff shape
   $\mathcal{W}$ with respect to the origin $o\in \mathbb{R}^{n+1}$ are related by  $${f}_{z_0}(x)=f(x)-\langle z_0,x\rangle.$$
   Then \eqref{s2.1} implies that
   \begin{equation*}
     \langle f_{z_0}(x)x+\bar{\nabla} f_{z_0}(x),e_{n+1}\rangle\leq 0
   \end{equation*}
   whenever $\langle x,e_{n+1}\rangle\leq 0$ for $x\in \mathbb{S}^n$.  This completes the proof of Proposition \ref{shift point}.
  \end{proof}

From now on, for a given smooth closed and uniformly convex Wulff shape $\mathcal{W}$, we always choose the point $z_0\in \widehat{\mathcal{W}}$ such that the support function $f_{z_0}$ with respect to $z_0$ satisfies the property in Proposition \ref{shift point}. For simplicity of the notations, we denote $f(x)=f_{z_0}(x)$ in the rest of the paper. So we have
  \begin{equation*}
    \langle {f}(x)x+\bar{\nabla} {f}(x),e_{n+1}\rangle\leq 0
  \end{equation*}
  whenever $\langle x,e_{n+1}\rangle\leq 0$ for $x\in \mathbb{S}^n$. We shall consider the anisotropic $\alpha$-Gauss curvature flow  associated with this anisotropy $f$. Note that when $f$ is a constant, we can rescale the constant to be $1$ and the corresponding Wulff shape $\mathcal{W}$ reduces to the unit round sphere $\mathbb{S}^n$. In this case, the flow \eqref{flow} reduces to the isotropic $\alpha$-Gauss curvature flow.

\subsection{The anisotropic Gauss curvature flow}
  Given a complete, non-compact, strict convex hypersurface $\Sigma_0$ embedded in $\mathbb{R}^{n+1}$, we let $F_0: M^n\to\mathbb{R}^{n+1}$  be an immersion with $F_0(M^n)=\Sigma_0$. By Theorem \ref{cpt sup}, $\Sigma_0$ is a graph over a convex domain $\Omega\subset\mathbb{R}^n$. Let $e_{n+1}$ be the direction satisfying
\begin{equation}\label{s2.2-1}
  \langle\nu(p), e_{n+1}\rangle\leq 0,\ \forall p\in M^n,
\end{equation}
  where $\nu(p)$ is the unit normal vector of $\Sigma_0$ at the point $F_0(p)$ pointing outward of the convex hull of $\Sigma_0$. Then by Proposition \ref{shift point}, the anisotropic normal vector field $\nu_f:=f(\mathbf{\nu})\mathbf{\nu}+\bar{\nabla} f(\mathbf{\nu})$  satisfies
   \begin{equation}\label{s2.vf}
     \langle \nu_f(p),e_{n+1}\rangle \leq 0,\ \forall p\in M^n.
   \end{equation}
   Based on this observation, it's natural to modify the flow \eqref{flow} by changing the direction of the evolution from the unit normal $\nu$ to the anisotropic unit normal $\nu_f$. Note that the vector field $\bar{\nabla}f(\nu)$ is a tangent vector on $T_{\nu}\mathbb{S}^n$, which could be viewed as a tangent vector at $T_p\Sigma_t$ since we have a natural identification between $T_p\Sigma_t$ and $T_{\nu(p)}\mathbb{S}^n$. This follows from the fact that the unit normal of $\mathbb{S}^n$ at $\nu(p)\in \mathbb{S}^n$ is just $\nu(p)$ itself, and the orthonormal frame of $T_p\Sigma_t$ also spans the tangent space $T_{\nu(p)}\mathbb{S}^n$ at $\nu(p)$. Let $g_{ij}$ denote the induced metric on $\Sigma_t$.  In local coordinates, we can write the tangent vector $\bar{\nabla}f(\nu)\in T_{\nu(p)} \mathbb{S}^n$ as $\bar{\nabla}f
(\nu)=f_k(\nu)g^{k\ell}\partial_\ell F$.

   Since $\bar{\nabla} f(\nu)$ is a tangent vector field of $\Sigma_t$, the modification of changing the direction $\nu$ to $\nu_f$ does not change the behavior of the flow and it's just a time-dependent tangential diffeomorphism. So in the following, we consider the following equivalent form of the flow \eqref{flow}
  \begin{equation}\label{flow1}
  	\left\{
  	\begin{aligned}
  		&\frac{\partial}{\partial t}F(p,t)=-K^{\alpha}(p,t)\left(f(\mathbf{\nu})\mathbf{\nu}+\bar{\nabla} f(\mathbf{\nu})\right),\\
  		&\lim_{t\to0}F(p,t)=F_0(p,t),
  	\end{aligned}
  	\right.
  \end{equation}
  where $K(p,t)$ is the Gauss curvature of $\Sigma_t$ at $F(p,t)$, $\nu(p,t)$ is the unit normal vector of $\Sigma_t$ at the point
  $F(p,t)$ pointing outward of the convex hull of $\Sigma_t$. On the other hand, by writing $\Sigma_t$ as graph of a convex function $u(\cdot,t)$, we can also work with an equivalent scalar parabolic equation for the function $u$:
\begin{equation}\label{s2.flowu}
  \frac{\partial }{\partial t}u=\rho(Du)\frac{(\mathrm{det}~ D^2u)^{\alpha}}{(1+|Du|^2)^{\frac{\alpha(n+2)-1}{2}}},
\end{equation}
where
\begin{equation*}
 \rho(Du)=f\left(\frac{\sum_{i=1}^ne_iD_iu-e_{n+1}}{\sqrt{1+|Du|^2}}\right).
\end{equation*}

\section{Evolution equations}\label{sec.evl}
In this section, we derive the evolution equations along the modified anisotropic $\alpha$-Gauss curvature flow \eqref{flow1}.

We first fix the notations. For a strictly convex smooth hypersurface $\Sigma$ in $\mathbb{R}^{n+1}$, we denote by $g_{ij}, h_{ij}$, $b^{ij}$ and $g^{ij}$ the induced metric, the second fundamental form, the inverse of the second fundamental form (namely $b^{ij}h_{jk}=\delta^i_k$) and the inverse matrix of $g_{ij}$. The mean curvature of $\Sigma$ is defined as $H=\sum_{i,j=1}^{n}g^{ij}h_{ij}$ and $K=\det(g^{ik}h_{kj})$ denotes the Gauss curvature of $\Sigma$. Let $\nabla$ denote the Levi-Civita connection with respect to the induced metric on $\Sigma$. We denote by $\mathcal{L}$ the linearized operator of the flow \eqref{flow}
      $$\mathcal{L} =\alpha fK^{\alpha}b^{ij}\nabla_i\nabla_j.$$
      Furthermore, $\langle,\rangle_{\mathcal{L} }$ denotes the associated inner product
      $$\langle \nabla g, \nabla h\rangle_{\mathcal{L} }=\alpha fK^{\alpha}b^{ij}\nabla_ig\nabla_jh,$$
      and $\Vert\cdot\Vert_{\mathcal{L} }$ denotes the $\mathcal{L}$-norm given by the inner product $\langle,\rangle_{\mathcal{L}}$. To obtain the $C^1$ estimate, it is convenient to consider the gradient function defined by
      \begin{equation}\label{s2.v}
        \upsilon =-\langle \nu, e_{n+1}\rangle^{-1}=\sqrt{1+|Du|^2}
      \end{equation}
       where $\nu$ is the unit normal vector field pointing outward of the convex hull of $\Sigma$. The upper bound of $\upsilon$ also ensures that the evolving hypersurface $\Sigma_t$ remains as a graph.

 We summarize the evolution equations in Proposition \ref{s2.pro2.4}. The calculation is standard (see e.g. \cite{andrews2020extrinsic}) but we need to carefully deal with the extral terms produced by the anisotropy $f$.
  \begin{proposition}\label{s2.pro2.4}
      Let $\Sigma_t$ be a smooth, complete strictly convex graph solution of (\ref{flow1}). Then the following evolution equations hold:
        \begin{align}
          \partial_t g_{ij}=&-2f(\nu)K^{\alpha}h_{ij}-\nabla_iK^{\alpha}f_j(\nu)-\nabla_jK^{\alpha}f_i(\nu)
          -K^{\alpha}f_{j\ell}h_{i}^\ell-K^{\alpha}f_{i\ell}h_{j}^\ell,\label{metric1}\\
         \partial_t \nu=&f(\nu)\nabla K^{\alpha},\label{normal vector}\\
          \partial_t h_{ij}=&\mathcal{L} h_{ij}+\alpha fK^{\alpha}\left(\alpha b^{k\ell}b^{mn}-b^{km}b^{\ell n}\right)\nabla_ih_{mn}\nabla_jh_{k\ell}
          +\alpha fK^{\alpha}Hh_{ij}\nonumber\\
          & -n\alpha fK^{\alpha}h_{im}h^{m}_j-K^{\alpha}(f_{k\ell}+fg_{k\ell})h_{j}^\ell h_{i}^k,\label{2nd fundamental form1}\\
          \partial_t K^{\alpha}=&\mathcal{L} K^{\alpha}+\alpha fK^{2\alpha}H+\alpha K^{2\alpha}f_{k\ell}g^{jk}h_{j}^\ell+2\alpha K^{\alpha}\langle\bar{\nabla} f, \nabla K^{\alpha}\rangle,\label{curvature}\\
           \partial_t(f(\nu)K^{\alpha})=&\mathcal{L}(f(\nu)K^{\alpha})+\alpha f^2K^{2\alpha}H,\label{speed function}\\
            \partial_t \upsilon =& \mathcal{L} \upsilon-2\upsilon^{-1}\Vert \nabla\upsilon\Vert^2_{\mathcal{L}}-\alpha fK^{\alpha}H\upsilon,\label{gradient function}
      \end{align}
  where $f_i, f_{ij}$ are components of $\bar{\nabla}f$ and the Hessian of $f$ with respect to the basis $\{\partial_iF,~i=1,\cdots,n\}$ of the tangent space $T_p\Sigma_t$, and the upper indices such as $h_i^\ell, h^{j\ell}$ are lifted using $g^{ij}$.
  \end{proposition}
  \begin{proof}[Proof of (\ref{metric1})]
  In local coordinates, the induced metric has components $g_{ij}=\langle \nabla_iF,\nabla_jF\rangle$. Then
      \begin{align}\label{s3.1}
          \partial_tg_{ij}=&\langle \partial_t\nabla_iF, \nabla_jF \rangle +\langle \nabla_iF,\partial_t\nabla_jF\rangle
          = \langle \nabla_i\partial_tF, \nabla_jF \rangle +\langle \nabla_iF,\nabla_j\partial_tF\rangle.
      \end{align}
      Along the flow \eqref{flow1}, we have
      \begin{eqnarray}\label{s3.2}
        \langle \nabla_i\partial_tF, \nabla_jF \rangle&=&
        \left\langle -\nabla_i\Big(f(\nu)K^{\alpha}\nu+K^{\alpha}\bar{\nabla} f(\nu)\Big), \nabla_jF \right\rangle\nonumber\\
        &=&-\Big\langle f(\nu)K^{\alpha}\nabla_i\nu+\nabla_iK^{\alpha}\bar{\nabla} f(\nu)
        +K^{\alpha}\nabla_i\left(\bar{\nabla} f(\nu)\right), \nabla_jF \Big\rangle.
           \end{eqnarray}
      By Weingarten equation $\langle \nabla_i\nu,\nabla_jF\rangle =h_{ij}$. We also have $\langle \bar{\nabla} f,\nabla_jF\rangle =f_j$. For the third term on the right hand side of \eqref{s3.2}, since $\bar{\nabla} f(\nu)=f_{k}(\nu)\partial_\ell F g^{k\ell}$ we have
      \begin{align*}
       \langle \nabla_i\left(\bar{\nabla} f(\nu)\right), \nabla_jF \rangle =& \langle\nabla_i(f_{k}(\nu)\partial_\ell F g^{k\ell}), \nabla_jF\rangle \\
        =& \nabla_i(f_k(\nu))g^{k\ell}g_{\ell j}\\
        =&\langle\bar{\nabla}(f_{j}(\nu)),\nabla_i\nu\rangle \\
        =&f_{j\ell}h_{i}^\ell.
      \end{align*}
      It follows that
       \begin{eqnarray}\label{s3.2a}
        \langle \nabla_i\partial_tF, \nabla_jF \rangle&=&
      -f(\nu)K^{\alpha}h_{ij}-\nabla_iK^{\alpha}f_j(\nu)-K^{\alpha}f_{j\ell}h_{i}^\ell.
      \end{eqnarray}
      Substituting \eqref{s3.2a} into \eqref{s3.1} gives the equation \eqref{metric1}.
  \end{proof}
  \begin{proof}[Proof of (\ref{normal vector})]
We compute that
      \begin{eqnarray*}
         \langle \partial_t\nu,\nabla_iF\rangle &=&-\langle \nu,\nabla_i\partial_tF\rangle \\
          &=&\left\langle \nu,\nabla_i\left(f(\nu)K^{\alpha}\nu+K^{\alpha}\bar{\nabla} f(\nu)\right)\right\rangle\\
          &=&\nabla_i\left(f(\nu)K^{\alpha}\right)+
          K^{\alpha}\langle \nu,\nabla_i\left(\bar{\nabla} f(\nu)\right) \rangle\\
          &=&f(\nu)\nabla_i K^{\alpha}+K^{\alpha}f_k(\nu)h_{i}^k-K^{\alpha}f_k(\nu)h_{i\ell}g^{k\ell}\\
          &=&f(\nu)\nabla_i K^{\alpha}.
      \end{eqnarray*}
Since $\vert\nu\vert^2=1$ implies $\langle \partial_t\nu,\nu\rangle=0$. We obtain \eqref{normal vector}.
  \end{proof}
  \begin{proof}[Proof of (\ref{2nd fundamental form1})]
      By definition, we have
      \begin{equation*}
        \nabla_i\nabla_jF=-h_{ij}\nu,
      \end{equation*}
      so that
      \begin{equation*}
        h_{ij}=-\langle \nabla_i\nabla_jF,\nu\rangle.
      \end{equation*}
      Along the flow \eqref{flow1}, we compute that
      \begin{eqnarray}\label{eq-4}
          \partial_th_{ij}&=&-\partial_t\langle\nabla_i\nabla_jF,\nu\rangle\nonumber\\
          &=&-\langle\nabla_i\nabla_j\partial_tF,\nu\rangle-\langle\nabla_i\nabla_jF,\partial_t\nu\rangle\nonumber\\
          &=&\Big\langle\nabla_i\nabla_j\left(f(\nu)K^{\alpha}\nu+K^{\alpha}\bar{\nabla} f(\nu)\right),\nu\Big\rangle
          +\left\langle h_{ij}\nu, \partial_t\nu\right\rangle\nonumber\\
          &=&\nabla_i\nabla_j(f(\nu))K^\alpha+\nabla_i(f(\nu))\nabla_jK^\alpha+\nabla_j(f(\nu))\nabla_iK^\alpha+f\nabla_i\nabla_jK^\alpha\nonumber\\
          &&+fK^\alpha\langle \nabla_i\nabla_j\nu,\nu\rangle +\nabla_iK^\alpha\langle \nabla_j(\bar{\nabla} f(\nu)),\nu\rangle +\nabla_jK^\alpha\langle \nabla_i(\bar{\nabla} f(\nu)),\nu\rangle\nonumber\\
          &&+K^\alpha\langle\nabla_i\nabla_j(\bar{\nabla} f(\nu)),\nu\rangle.
      \end{eqnarray}
      $\langle\nabla_i\nu,\nu\rangle=0$ implies
      \begin{equation}\label{eq-5}
         \langle\nabla_i\nabla_j\nu,\nu\rangle=-\langle\nabla_i\nu,\nabla_j\nu\rangle=-h_{i\ell}h^{\ell}_j.
      \end{equation}
     For the derivatives of the anisotropy $f$, we have
      \begin{align}\label{eq-6}
        \langle\nabla_j\left(\bar{\nabla} f(\nu)\right),\nu\rangle=&
        -\langle\bar{\nabla} f(\nu),\nabla_j\nu\rangle\nonumber\\
        =&-\langle f_k\partial_\ell Fg^{k\ell},\nabla_j\nu\rangle \nonumber\\
        =&-f_kh_j^k\nonumber\\
        =&-\nabla_j\left(f(\nu)\right).
      \end{align}
      Similarly,
     \begin{eqnarray}\label{eq-8}
        \langle\nabla_i\nabla_j\left(\bar{\nabla} f(\nu)\right),\nu\rangle&=&
        \nabla_i\langle\nabla_j\left(\bar{\nabla} f(\nu)\right),\nu\rangle-
        \langle\nabla_j\left(\bar{\nabla} f(\nu)\right),\nabla_i\nu\rangle\nonumber\\
        &=&-\nabla_i\nabla_j(f(\nu))-\langle\nabla_j\left(f_k(\nu)\partial_\ell F g^{k\ell}\right),\nabla_i\nu\rangle\nonumber\\
        &=&-\nabla_i\nabla_j(f(\nu))-f_{k\ell}h_{j}^\ell h_{i}^k.
      \end{eqnarray}
      Now, combining \eqref{eq-4} -- \eqref{eq-8}, we deduce that
      \begin{eqnarray}\label{s3.h}
        \partial_th_{ij}&=&f\nabla_i\nabla_jK^{\alpha}-K^{\alpha}(f_{k\ell}+fg_{k\ell})h_{j}^\ell h_{i}^k.
      \end{eqnarray}
      Next,
      \begin{eqnarray*}
        f\nabla_i\nabla_jK^{\alpha}&=&f\nabla_i(\alpha K^{\alpha}b^{mn}h_{mn,j})\\
        &=&\alpha^2fK^{\alpha}b^{kl}b^{mn}h_{kl,i}h_{mn,j}+\alpha fK^{\alpha}\nabla_ib^{mn}h_{mn,j}+\alpha fK^{\alpha}b^{mn}h_{mn,ji}\\
        &=&\alpha fK^{\alpha}b^{mn}h_{mn,ji}+\alpha^2fK^{\alpha}b^{kl}b^{mn}h_{kl,i}h_{mn,j}-\alpha fK^{\alpha}b^{km}b^{ln}h_{kl,i}h_{mn,j}\\
        &=&\alpha fK^{\alpha}b^{mn}h_{mn,ji}+\alpha fK^{\alpha}\left(\alpha b^{kl}b^{mn}-b^{km}b^{ln}\right)h_{kl,i}h_{mn,j}.
      \end{eqnarray*}
  By the similar computation used in the proof of Simons' identity (see \cite[Chapter 2.1]{colding2011course}),
      \begin{eqnarray*}
        \alpha fK^{\alpha}b^{mn}h_{mn,ji}&=&\alpha fK^{\alpha}b^{mn}h_{mj,ni}\\
        &=&\alpha fK^{\alpha}b^{mn}h_{mj,in}+\alpha fK^{\alpha}b^{mn}R_{nikj}h^{k}_m+\alpha fK^{\alpha}b^{mn}R_{nikm}h^{k}_j\\
        &=&\alpha fK^{\alpha}b^{mn}h_{ij,mn}+\alpha fK^{\alpha}b^{mn}(h_{nk}h_{ij}-h_{nj}h_{ik})h^{k}_m\\
        &&+\alpha fK^{\alpha}b^{mn}(h_{nk}h_{im}-h_{nm}h_{ik})h^{k}_j\\
        &=&\mathcal{L}h_{ij}+\alpha fK^{\alpha}(h_{nk}h_{ij}-h_{nj}h_{ik})g^{kn}-(n-1)\alpha fK^{\alpha}h_{ik}h^{k}_j\\
        &=&\mathcal{L}h_{ij}+\alpha fK^{\alpha}Hh_{ij}-n\alpha fK^{\alpha}h_{ik}h^{k}_j.
      \end{eqnarray*}
      Therefore
      \begin{eqnarray}\label{s3.h2}
        f\nabla_i\nabla_jK^{\alpha}&=&\mathcal{L} h_{ij}
        +\alpha fK^{\alpha}\left(\alpha b^{kl}b^{mn}-b^{km}b^{ln}\right)\nabla_ih_{mn}\nabla_jh_{kl}\nonumber\\
        &&+\alpha fK^{\alpha}Hh_{ij}-n\alpha fK^{\alpha}h_{im}h^{m}_j.
      \end{eqnarray}
      Substituting \eqref{s3.h2} into \eqref{s3.h}, we obtain \eqref{2nd fundamental form1}.
  \end{proof}
  \begin{proof}[Proof of (\ref{curvature})]
     By definition $K=\det(h_{ij})\det(g^{ij})$, using \eqref{s3.h},\eqref{metric1} and $g^{ij}g_{jk}=\delta_k^i$ we derive that
      \begin{eqnarray*}
          \partial_tK^{\alpha}&=&\alpha K^{\alpha}\partial_t\left(\log(\det h_{ij})+\log(\det g^{ij})\right)\\
         & =&\alpha K^{\alpha}b^{ij}\partial_th_{ij}+\alpha K^{\alpha}g_{ij}\partial_tg^{ij}\\
           &=&\alpha K^{\alpha}b^{ij}\partial_th_{ij}-\alpha K^{\alpha}g^{ij}\partial_tg_{ij}\\
          &=&\alpha K^{\alpha}b^{ij}\Big(f\nabla_i\nabla_jK^{\alpha}-K^{\alpha}(f_{k\ell}+fg_{k\ell})h_{j}^\ell h_{i}^k\Big)
          +\alpha K^{\alpha}g^{ij}\Big(2f(\nu)K^{\alpha}h_{ij}\\
          &&+\nabla_iK^{\alpha}f_j(\nu)+\nabla_jK^{\alpha}f_i(\nu)
        +K^{\alpha}f_{j\ell}(\nu)h_i^{\ell}+K^{\alpha}f_{i\ell}(\nu)h_j^{\ell}\Big)\\
          &=&\mathcal{L} K^{\alpha}+\alpha fK^{2\alpha}H+\alpha K^{2\alpha}f_{k\ell}g^{jk}h_{j}^\ell+2\alpha K^{\alpha}\langle\nabla f, \nabla K^{\alpha}\rangle.
      \end{eqnarray*}
  \end{proof}

  \begin{proof}[Proof of (\ref{speed function})]
 Firstly, by \eqref{normal vector},
      \begin{eqnarray*}
          \partial_tf(\nu)=\langle\bar{\nabla}f, \partial_t\nu\rangle=f\langle\bar{\nabla} f,\nabla K^{\alpha}\rangle.
      \end{eqnarray*}
    Meanwhile,
    \begin{align*}
      \nabla_kK^{\alpha}=&\alpha K^{\alpha-1}\nabla_k(\det(g^{ij})\det(h_{ij})) \\
     =&\alpha K^{\alpha-1}\det(g^{ij})\det(h_{ij})b^{mn}h_{mn,k}=
    \alpha K^{\alpha}b^{ij}h_{ij,k}.
    \end{align*}
    Then
    \begin{eqnarray*}
        \mathcal{L}(f(\nu))&=&\alpha fK^{\alpha}b^{ij}\nabla_i\nabla_j(f(\nu))=\alpha fK^{\alpha}b^{ij}\nabla_i(f_kh_{j}^k)\\
        &=&\alpha fK^{\alpha}b^{ij}f_{k\ell}h_{i}^\ell h_{j}^k+\alpha fK^{\alpha}b^{ij}f_k\nabla_ih_{j}^k\\
        &=&\alpha fK^{\alpha}f_{k\ell}g^{j\ell}h_{j}^k+f\langle\bar{\nabla} f,\nabla K^{\alpha}\rangle,
    \end{eqnarray*}
    where we used the Codazzi equation $\nabla_ih_{jk}=\nabla_kh_{ij}$. This gives
    \begin{equation}\label{anisotropic function}
      \partial_t f(\nu)=\mathcal{L}f(\nu)-\alpha fK^{\alpha}f_{k\ell}g^{j\ell}h_{j}^k.
    \end{equation}
      Combining \eqref{anisotropic function} with \eqref{curvature} yields
      \begin{eqnarray*}
          \partial_t(f(\nu)K^{\alpha})&=&f(\nu)\partial_tK^{\alpha}+\partial_t(f(\nu))K^{\alpha}\\
          &=&f\left(\mathcal{L} K^{\alpha}+\alpha fK^{2\alpha}H+\alpha K^{2\alpha}f_{k\ell}g^{jk}h_{j}^\ell+2\alpha K^{\alpha}\langle\bar{\nabla} f, \nabla K^{\alpha}\rangle\right)\\
          &&+K^{\alpha}\left(\mathcal{L}f-\alpha fK^{\alpha}f_{k\ell}g^{j\ell}h_{j}^k\right)\\
          &=&\mathcal{L}(fK^{\alpha})-2\alpha f K^{\alpha}b^{ij}\nabla_i(f(\nu))\nabla_jK^{\alpha}+\alpha f^2K^{2\alpha}H\\
          &&+2\alpha fK^{\alpha}\langle\bar{\nabla} f, \nabla K^{\alpha}\rangle\\
          &=&\mathcal{L}(fK^{\alpha})-2\alpha f K^{\alpha}b^{ij}f_kh_{i}^k\nabla_jK^{\alpha}+\alpha f^2K^{2\alpha}H\\
          &&+2\alpha fK^{\alpha}\langle\bar{\nabla} f, \nabla K^{\alpha}\rangle\\
          &=&\mathcal{L}(fK^{\alpha})+\alpha f^2K^{2\alpha}H.
      \end{eqnarray*}
  \end{proof}
\begin{proof}[Proof of (\ref{gradient function})]
By \eqref{normal vector} we have
      \begin{equation}\label{s3.v1}
          \partial_t\upsilon=\upsilon^2\langle\partial_t\nu, e_{n+1}\rangle=f\upsilon^2\langle \nabla K^{\alpha}, e_{n+1}\rangle.
      \end{equation}
    On the other hand,
    \begin{eqnarray}\label{s3.v2}
        \mathcal{L} \upsilon&=&\alpha fK^{\alpha}b^{ij}\nabla_i\nabla_j\upsilon
        =\alpha fK^{\alpha}b^{ij}\nabla_i\left(\upsilon^2\langle \nabla_j\nu, e_{n+1}\rangle \right)\nonumber\\
        &=&2\upsilon^{-1}\alpha fK^{\alpha}b^{ij}\nabla_i\upsilon\nabla_j\upsilon+\alpha fK^{\alpha}b^{ij}\upsilon^2
        \nabla_i\langle h_{jk}\nabla^kF, e_{n+1}\rangle\nonumber\\
        &=&2\upsilon^{-1}\alpha fK^{\alpha}b^{ij}\nabla_i\upsilon\nabla_j\upsilon+\upsilon^2\langle\alpha fK^{\alpha}b^{ij}h_{ij,k}\nabla^kF,e_{n+1}\rangle\nonumber\\
        &&   -\alpha fK^{\alpha}b^{ij}\upsilon^2h_{jk}h^k_i\langle\nu,e_{n+1}\rangle\nonumber\\
        &=&2\upsilon^{-1}\Vert \nabla\upsilon\Vert^2_{\mathcal{L}}
        +f\upsilon^2\langle e_{n+1}, \nabla K^{\alpha}\rangle+\alpha fK^{\alpha}H\upsilon.
    \end{eqnarray}
    Combining \eqref{s3.v1} and \eqref{s3.v2} gives the equation \eqref{gradient function}.
  \end{proof}

  \section{Local gradient estimate}\label{sec2}
  In this section, we derive the local estimate for the gradient function $\upsilon$ defined in \eqref{s2.v}. The proof relies on choosing a suitable cut-off function, and the property of the support function $f$ in Proposition \ref{shift point} plays a crucial role during the estimate.

  Assume that $\Sigma_t$ is a strictly convex graph solution of \eqref{flow1}. Let $\bar{u}(p,t)=\langle F(p,t),e_{n+1}\rangle$ denote the height function. Given constants $N>0$ and $\beta\geq0$, we define the cut-off function $\psi_{\beta}$ by
  \begin{equation}\label{s4.cutf}
    \psi_{\beta}(p,t)=(N-\beta t-\bar{u}(p,t))_+=\max\{N-\beta t-\bar{u}(p,t),0\}.
  \end{equation}
 \begin{lemma}
 	Along the flow (\ref{flow1}), the cut-off function $\psi_{\beta}$ evolves by
  \begin{equation}\label{cut-off fun1}
    \partial_t\psi_{\beta}(p,t)=\mathcal{L}\psi_{\beta}(p,t)+n\alpha fK^{\alpha}\upsilon^{-1}
    +K^{\alpha}\langle f(\nu)\nu+\bar{\nabla} f(\nu), e_{n+1}\rangle-\beta.
  \end{equation}
\end{lemma}
   \begin{proof}
    By Theorem \ref{cpt sup}, $\psi_{\beta}$ is compactly supported and on its support we have
    \begin{eqnarray*}
      \partial_t\psi_{\beta}(p,t)=-\beta -\langle\partial_tF,e_{n+1}\rangle=-\beta+K^{\alpha}\langle f(\nu)\nu+\bar{\nabla} f(\nu), e_{n+1}\rangle.
    \end{eqnarray*}
    Moreover,
    \begin{eqnarray*}
      \mathcal{L}\psi_{\beta}(p,t)&=&\alpha fK^{\alpha}b^{ij}\nabla_i\nabla_j\psi_{\beta}(p,t)=-\alpha fK^{\alpha}b^{ij}\langle\nabla_i\nabla_jF,e_{n+1}\rangle\\
      &=&\alpha fK^{\alpha}b^{ij}h_{ij}\langle \nu, e_{n+1}\rangle=n\alpha fK^{\alpha}\langle \nu, e_{n+1}\rangle.
    \end{eqnarray*}
The equation \eqref{cut-off fun1} follows immediately.
 \end{proof}
 \begin{lemma}\label{gradient-est1}
 	Along the flow (\ref{flow1}), the following holds:
   \begin{eqnarray*}
     \partial_t(\psi_{\beta}\upsilon)
     &\leq&\mathcal{L}(\psi_{\beta}\upsilon)-2\alpha fK^{\alpha}b^{ij}\upsilon^{-1}\nabla_i(\psi_{\beta}\upsilon)\nabla_j\upsilon+n\alpha fK^{\alpha}-\beta\upsilon\\
     &&-\alpha fK^{\alpha}H\psi_{\beta}\upsilon.
   \end{eqnarray*}
 \end{lemma}
 \begin{proof}
   By (\ref{gradient function}) and (\ref{cut-off fun1}), we have
   \begin{eqnarray*}
    \partial_t(\psi_{\beta}\upsilon)-\mathcal{L}(\psi_{\beta}\upsilon)&=&\upsilon\left(\partial_t\psi_{\beta}-\mathcal{L}\psi_\beta\right)+\psi_{\beta}\left(\partial_t\upsilon- \mathcal{L}\upsilon\right)-2\alpha fK^\alpha b^{ij}\nabla_i\psi_\beta\nabla_j\upsilon\\
    &=&\upsilon\Big(n\alpha fK^{\alpha}\upsilon^{-1}
    +K^{\alpha}\langle f(\nu)\nu+\bar{\nabla} f(\nu), e_{n+1}\rangle-\beta\Big)\\
    &&-\psi_{\beta}\left(2\upsilon^{-1}\Vert \nabla\upsilon\Vert^2_{\mathcal{L}}+\alpha fK^{\alpha}H\upsilon\right)-2\alpha fK^\alpha b^{ij}\nabla_i\psi_\beta\nabla_j\upsilon\\
    &=&-2\alpha fK^{\alpha}b^{ij}\upsilon^{-1}\nabla_i(\psi_{\beta}\upsilon)\nabla_j\upsilon+n\alpha fK^{\alpha}-\beta\upsilon\\
    &&-\alpha fK^{\alpha}H\psi_{\beta}\upsilon+K^{\alpha}\langle  f(\nu)\nu+\bar{\nabla} f(\nu), e_{n+1}\rangle\upsilon.
   \end{eqnarray*}
    By Proposition \ref{shift point}, $\upsilon=-\langle\nu, e_{n+1}\rangle^{-1}\geq 0$ leads to $\langle  f(\nu)\nu+\bar{\nabla} f(\nu), e_{n+1}\rangle\leq 0$,
    which implies the lemma.
 \end{proof}
\begin{theorem}[Gradient estimate]\label{Gradient estimate}
  Let $\Sigma_t$ be a complete strictly convex smooth graph solution of (\ref{flow1}) defined on $M^n\times[0,T]$, for some $T>0$. Given
  constants $\beta>0$ and $N\geq \beta$,
  \begin{eqnarray}\label{s4.est}
    (\psi_{\beta}\upsilon)(p,t)\leq \max\Big\{N\sup_{\mathcal{Q}_N}\upsilon(p,0),{N}{\beta}^{-1}n\alpha \max\{\sup_{\mathbb{S}^n}f,1\}\Big\},\quad \forall~(p,t)\in M^n\times [0,T],
  \end{eqnarray}
  where $\mathcal{Q}_N=\{p\in M^n:\bar{u}(p,0)<N\}$.
\end{theorem}
\begin{proof}
    Since the cut-off function $\psi_{\beta}$ is compactly supported, for fixed $T\in(0,+\infty)$, the function $\psi_{\beta}\upsilon$ attains its maximum on
    $M^n\times[0,T]$ at some point $(p_0,t_0)$. If $t_0=0$, then
    \begin{equation}\label{s3.g1}
      (\psi_{\beta}\upsilon)(p,t)\leq  \sup_{p\in M}(\psi_{\beta}\upsilon)(p,0)\leq N\sup_{\mathcal{Q}_N}\upsilon(p,0)
    \end{equation}
    for all $p\in M^n$.  Assume that $t_0>0$, then by Lemma \ref{gradient-est1}, at $(p_0,t_0)$
    we have
    \begin{eqnarray*}
     0\leq\partial_t(\psi_{\beta}\upsilon)-\mathcal{L}(\psi_{\beta}\upsilon)
     &\leq&n\alpha fK^{\alpha}-\beta\upsilon-\alpha fK^{\alpha}H\psi_{\beta}\upsilon,
    \end{eqnarray*}
    which is equivalent to
    \begin{eqnarray}\label{s4.1}
      n\alpha fK^{\alpha}\geq\beta\upsilon+\alpha fK^{\alpha}H\psi_{\beta}\upsilon.
    \end{eqnarray}
    Using the facts that $\psi_\beta\leq N$, $\beta\leq N$, $H\geq nK^{\frac{1}{n}}$ and $f$ is positive and bounded from above, we rearrange \eqref{s4.1} and obtain
    \begin{align}
     n\alpha\geq&~\beta\upsilon f^{-1}K^{-\alpha}+\alpha H\psi_{\beta}\upsilon\nonumber\\
     \geq&~ \frac{\beta}{N} f^{-1} K^{-\alpha}\psi_{\beta}\upsilon+\alpha H\psi_{\beta}\upsilon\nonumber\\
     \geq&~ \frac{\beta}{N}\min\{1/(\sup_{\mathbb{S}^n}f),1\}\left(K^{-\alpha}+n\alpha K^{1/n}\right)\psi_{\beta}\upsilon,\label{s4.2}
    \end{align}
    Applying Young's inequality
    \begin{align*}
     K^{-\alpha}+n\alpha K^{1/n}\geq & \frac{1}{1+n\alpha}K^{-\alpha}+\frac{n\alpha}{1+n\alpha} K^{1/n} ~\geq 1,
    \end{align*}
   then \eqref{s4.2} implies that
   \begin{equation}\label{s3.g2}
    \psi_{\beta}\upsilon\leq ~{N}{\beta}^{-1}n\alpha \max\{\sup_{\mathbb{S}^n}f,1\}.
   \end{equation}
   The estimate \eqref{s4.est} then follows from combining \eqref{s3.g1} and \eqref{s3.g2}.
  \end{proof}

  \section{Lower bound on the principal curvature}\label{sec3}
  In this section, we estimate the lower bound on the smallest principal curvature $\lambda_{\min}$.
    Since the minimum eigenvalue is not necessarily smooth, we adopt the following derivatives of a smooth approximation for the minimum eigenvalue, which was established by Brendle, Choi and Daskalopoulos \cite{brendle2017asymptotic} firstly.
  \begin{lemma}\cite[Lemma 4.1]{choi2023LpMinkowski}\label{approx}
  	Let $\mu$ be the multiplicity of the smallest principal curvature at a point $p_0$ on $\Sigma_{t_0}$ for $t_0>0$ so that $\lambda_1=\cdots=\lambda_{\mu}<\lambda_{\mu+1}\leq\cdots\leq\lambda_n$, where $\lambda_1,\cdots,\lambda_n$ are the principal curvatures. Suppose $\phi$ is a smooth function defined on $\mathcal{M}=\cup_{0<t\leq t_0}\Sigma_t\times\{t\}$ such that $\phi\leq\lambda_1$ on $\mathcal{M}$ and $\phi=\lambda_1$ at the point $(p_0,t_0)$. Then, at this point with a chart $g_{ij}=\delta_{ij}$, $h_{ij}=\delta_{ij}\lambda_i$, we have
  	\begin{align}\label{approx-eq0}
  		\nabla_ih_{k\ell}=\nabla_i\phi\delta_{k\ell},\quad \text{for}\ 1\leq k,\ell\leq\mu,
  	\end{align}
  	\begin{align}\label{approx-eq}
  		\nabla_i\nabla_i\phi\leq\nabla_i\nabla_ih_{11}-2\sum_{\ell>\mu}\frac{(\nabla_ih_{1\ell})^2}{\lambda_{\ell}-\lambda_1},
  	\end{align}
  	and
  	\begin{align}\label{approx-eq1}
  		\partial_t\phi\geq\partial_th_{11}-\partial_tg_{11}\lambda_1.
  	\end{align}
  \end{lemma}
\begin{lemma}
  Let $\Sigma_t$ be a complete strictly convex smooth graph solution of (\ref{flow1}) defined on $M^n\times[0,T]$, for some $T>0$.  Assume that $\phi$ is the function satisfying the conditions in Lemma \ref{approx}. Then, the following holds at the point $(p_0,t_0)$:
  \begin{align}\label{s4-eq0}
  	\partial_t\phi-\mathcal{L}\phi\geq &-n\alpha fK^{\alpha}\lambda_1^2+\alpha fK^{\alpha}H\lambda_1+2\alpha K^{\alpha}b^{rs}h_{rs,1}f_1\lambda_1+2\alpha fK^{\alpha}\sum_{i}\sum_{\ell>\mu}\frac{(\nabla_ih_{1\ell})^2}{\lambda_i(\lambda_{\ell}-\lambda_1)}\nonumber\\
  	&+\alpha fK^{\alpha}\left(\alpha b^{rs}b^{mn}-b^{rm}b^{sn}\right)\nabla_1h_{mn}\nabla_1h_{rs}
  \end{align}

\end{lemma}
\proof

Under the flow (\ref{flow1}), from the inequalities \eqref{approx-eq} and \eqref{approx-eq1}, we have
\begin{align}\label{s4-eq1}
		\partial_t\phi-\mathcal{L}\phi\geq &\partial_th_{11}-\mathcal{L}h_{11}-\partial_tg_{11}\lambda_1+2\alpha fK^{\alpha}\sum_{i}\sum_{\ell>\mu}\frac{(\nabla_ih_{1\ell})^2}{\lambda_i(\lambda_{\ell}-\lambda_1)}.
\end{align}
Substituting \eqref{metric1} and \eqref{2nd fundamental form1} into \eqref{s4-eq1} we get
\begin{align}
	\partial_t\phi-\mathcal{L}\phi\geq &\alpha fK^{\alpha}\left(\alpha b^{k\ell}b^{mn}-b^{km}b^{\ell n}\right)\nabla_1h_{mn}\nabla_1h_{k\ell}
	+\alpha fK^{\alpha}H\lambda_1\nonumber\\
	& -n\alpha fK^{\alpha}\lambda_1^2+K^{\alpha}(f_{11}+f)\lambda_1^2+2\alpha K^{\alpha}b^{rs}h_{rs,1}f_1\lambda_1\\
	&+2\alpha fK^{\alpha}\sum_{i}\sum_{\ell>\mu}\frac{(\nabla_ih_{1\ell})^2}{\lambda_i(\lambda_{\ell}-\lambda_1)}.\nonumber
\end{align}
Since the Wulff shape $\mathcal{W}$ is convex, we have $\bar{\nabla}^2f+fg_{\mathbb{S}^n}$ is uniformly positive definite. In particular, $f_{11}+f>0$ holds at the point $(p_0,t_0)$. This implies the inequality \eqref{s4-eq0}.
\endproof

To estimate the local lower bound of the smallest principal curvature, we multiply $\lambda_{\min}$ by suitable cut-off function and apply the maximum principle. The following is the main estimate of this section.
  \begin{theorem}\label{principle curvature}
    Let $\Sigma_t$ be a complete strictly convex smooth graph solution of (\ref{flow1}) defined on $M^n\times[0,T]$, for some $T>0$. Given
  constants $\beta>0$ and $N\geq \beta$,
  \begin{eqnarray}\label{s5.t1}
    (\psi_{\beta}^{-n(1+\frac{1}{\alpha})}\lambda_{\min})(p,t)\geq \min\left\{N^{-n(1+\frac{1}{\alpha})}\inf_{\mathcal{Q}_N}\{\lambda_{\min}(p,0)\}, 1/C_1(N,f,n,\alpha,\beta)\right\},
  \end{eqnarray}
  where $\psi_\beta$ is the cut-off function defined in \eqref{s4.cutf}, $\mathcal{Q}_N=\{p\in M^n:\bar{u}(p,0)<N\}$ and $C_1(N,f,n,\alpha,\beta)$ is the constant in \eqref{s4.w2} which depends only on $N,f,n,\alpha$ and $\beta$.
  \end{theorem}
  \begin{proof}
      Denote the auxiliary function
       \begin{equation*}
         \omega:=\psi_{\beta}^{\gamma}\frac{1}{\lambda_{\min}},
       \end{equation*}
       where $\gamma=n(1+\frac{1}{\alpha})>0$ is a positive constant. Since $\psi_{\beta}$ is compactly supported, for a fixed $T\in(0,+\infty)$, the function $\omega$ attains its maximum on $M^n\times[0,T]$
    at a point $(p_0,t_0)$. If $t_0=0$, by $\psi_\beta\leq N$ we have
 \begin{align}\label{s4.w1}
  \omega(p,t)\leq & \omega(p_0,0)\leq N^{n(1+\frac{1}{\alpha})}\sup_{\mathcal{Q}_N}\{\lambda^{-1}_{\min}(p,0)\},
 \end{align}
  where $\mathcal{Q}_N=\{p\in M^n:\bar{u}(p,0)<N\}$. So without loss of generality we assume $t_0>0$ in the following.

To estimate $\omega_0:=\omega(p_0,t_0)$, we define
  \begin{align}\label{s4-eq2}
  	\phi=\psi_{\beta}^{\gamma}\omega_0^{-1}.
  \end{align}
Then, observing $\omega\leq\omega_0$ we have
\begin{equation*}
  \phi\leq\lambda_{min}\quad  \mathrm{on} \quad \Sigma_t\cap \mathrm{supp}(\psi_\beta)
\end{equation*}
for $t\in[0,t_0]$ as well as $\phi(p_0,t_0)=\lambda_{\min}(p_0,t_0)$. Hence, we can apply the estimates \eqref{approx-eq}, \eqref{approx-eq1} and \eqref{s4-eq0} to the smooth function $\phi$, as they are calculated at the maximum point. We choose a chart satisfies $g_{ij}=\delta_{ij}$, $h_{ij}=\delta_{ij}\lambda_i$ at $(p_0,t_0)$, where $\lambda_1=\cdots=\lambda_{\mu}<\lambda_{\mu+1}\leq\cdots\leq\lambda_n$.

   \textbf{Step 1. Calculate the evolution equation of $\psi_{\beta}^{\gamma}\phi^{-1}$}.
    First, we note that on the support of $\psi_\beta$,
   \begin{align*}
     \partial_t\ln(\psi_{\beta}^{\gamma}\phi^{-1}) =&\gamma\psi_{\beta}^{-1} \partial_t\psi_{\beta}-\phi^{-1}\partial_t  \phi,\\
     \nabla_i\ln(\psi_{\beta}^{\gamma}\phi^{-1}) =&\gamma\psi_{\beta}^{-1}\nabla_i\psi_{\beta}-\phi^{-1}\nabla_i\phi,\\
      \mathcal{L}\ln(\psi_{\beta}^{\gamma}\phi^{-1})=&\alpha fK^{\alpha}b^{ij}\nabla_i\left(\gamma\psi_{\beta}^{-1}\nabla_j\psi_{\beta}-\phi^{-1}\nabla_j\phi\right)\\
      =&\gamma\psi_{\beta}^{-1}\mathcal{L}\psi_{\beta}-\phi^{-1}\mathcal{L}\phi-\gamma\psi_{\beta}^{-2}\Vert\nabla\psi_{\beta}\Vert^2_{\mathcal{L}}
   +\phi^{-2}\Vert\nabla \phi\Vert^2_{\mathcal{L}}.
   \end{align*}
 At the point $(p_0,t_0)$,  combining these equations with \eqref{s4-eq0} and \eqref{cut-off fun1} gives that
 \begin{align}\label{s5.2}
 (\partial_t-\mathcal{L})\ln(\psi_{\beta}^{\gamma}\phi^{-1})=&\gamma\psi_{\beta}^{-1}\left(\partial_t\psi_\beta-\mathcal{L}\psi_\beta\right)-\phi^{-1}\left(\partial_t  \phi-\mathcal{L}\phi\right)\nonumber\\
  &\quad +\gamma\psi_{\beta}^{-2}\Vert\nabla\psi_{\beta}\Vert^2_{\mathcal{L}}
   -\phi^{-2}\Vert\nabla\phi\Vert^2_{\mathcal{L}}\nonumber\\
  \leq &\gamma\psi_{\beta}^{-2}\Vert\nabla\psi_{\beta}\Vert^2_{\mathcal{L}}
  -\phi^{-2}\Vert\nabla \phi\Vert^2_{\mathcal{L}}+\gamma n\alpha fK^{\alpha}\upsilon^{-1}\psi_{\beta}^{-1}\nonumber\\
  &-\beta\gamma\psi_{\beta}^{-1}+n\alpha fK^{\alpha}\lambda_1-2\alpha K^{\alpha}b^{rs}h_{rs,1}f_1\nonumber\\
  &-\alpha fK^{\alpha}H-2\alpha fK^{\alpha}\sum_{i}\sum_{\ell>\mu}\frac{(\nabla_ih_{1\ell})^2}{\lambda_1\lambda_i(\lambda_{\ell}-\lambda_1)}\nonumber\\
  &-\alpha fK^{\alpha}\lambda_1^{-1}\left(\alpha b^{rs}b^{mn}-b^{rm}b^{sn}\right)\nabla_1h_{mn}\nabla_1h_{rs}.
 \end{align}
By the definition \eqref{s4-eq2}, we have $\ln(\psi_{\beta}^{\gamma}\phi^{-1})=\ln\omega_0$, then
\begin{align}\label{s4-eq3}
	\nabla\ln(\psi_{\beta}^{\gamma}\phi^{-1})=0, \text{and}\  (\partial_t-\mathcal{L})\ln(\psi_{\beta}^{\gamma}\phi^{-1})=0.
\end{align}
Hence for any $i=1,\cdots,n$,
 \begin{eqnarray*}
  \gamma\psi_{\beta}^{-1}\nabla_i\psi_{\beta}=\phi^{-1}\nabla_i\phi=\phi^{-1}\nabla_ih_{11},
 \end{eqnarray*}
 where in the second identity we used \eqref{approx-eq0}.
Then the first two terms on the third line of \eqref{s5.2} satisfy
 \begin{align}\label{s5.3}
\gamma\psi_{\beta}^{-2}\Vert\nabla\psi_{\beta}\Vert^2_{\mathcal{L}}
  -\phi^{-2}\Vert\nabla \phi\Vert^2_{\mathcal{L}}=&(\gamma^{-1}-1)\phi^{-2}\Vert\nabla h_{11}\Vert^2_{\mathcal{L}}\nonumber\\
  =&\alpha(\gamma^{-1}-1)\lambda_1^{-2}fK^{\alpha}\sum_ib^{ii}\vert\nabla_ih_{11}\vert^2.
 \end{align}
 Using the Cauchy-Schwarz inequality, we can estimate the third term of the fourth line of \eqref{s5.2} as
 \begin{align}\label{s5.4}
  -2\alpha K^{\alpha}b^{rs}h_{rs,1}f_1\leq  &  \epsilon^{-1}f^{-1}\vert f_1\vert^2K^{\alpha}\lambda_1+\alpha^2\epsilon fK^{\alpha}\lambda_1^{-1}\vert b^{rs}h_{rs,1}\vert^2,
 \end{align}
 where $\epsilon>0$ is a small constant to be determined later. Note that in the isotropic case this term vanishes because $f$ is a constant.  Substituting \eqref{s4-eq3}, \eqref{s5.3} and \eqref{s5.4} into \eqref{s5.2} implies that
 \begin{align}\label{s5.5}
0\leq&\gamma n\alpha fK^{\alpha}\upsilon^{-1}\psi_{\beta}^{-1}-\beta\gamma\psi_{\beta}^{-1}+n\alpha fK^{\alpha}\lambda_1-\alpha fK^{\alpha}H\nonumber\\
  & +\epsilon^{-1}f^{-1}\vert f_1\vert^2K^{\alpha}\lambda_1+\alpha(\gamma^{-1}-1)\lambda_1^{-2}fK^{\alpha}\sum_ib^{ii}\vert\nabla_ih_{11}\vert^2\nonumber\\
  & +\alpha^2\epsilon fK^{\alpha}\lambda_1^{-1}\vert b^{rs}h_{rs,1}\vert^2-\alpha fK^{\alpha}\lambda_1^{-1}\left(\alpha b^{rs}b^{mn}-b^{rm}b^{sn}\right)\nabla_1h_{mn}\nabla_1h_{rs}\nonumber\\
  &-2\alpha fK^{\alpha}\sum_{i}\sum_{\ell>\mu}\frac{(\nabla_ih_{1\ell})^2}{\lambda_1\lambda_i(\lambda_{\ell}-\lambda_1)}\nonumber\\
  &=:Q_0+Q_1,
 \end{align}
where $Q_0$ denotes the zero order terms on the right hand side of \eqref{s5.5} and $Q_1$ denotes the last four terms of \eqref{s5.5} which involve the derivatives of the second fundamental forms.

\textbf{Step 2. Estimate the $Q_1$ terms}. To estimate $\omega_0$ using  \eqref{s5.5}, we need to deal with the derivative terms $Q_1$.
 \begin{eqnarray}\label{s5.6}
Q_1&=&\alpha(\gamma^{-1}-1)\lambda_1^{-2}fK^{\alpha}\sum_ib^{ii}\vert\nabla_ih_{11}\vert^2+\alpha^2\epsilon fK^{\alpha}\lambda_1^{-1}\vert b^{rs}h_{rs,1}\vert^2\nonumber\\
  &&-\alpha fK^{\alpha}\lambda_1^{-1}\left(\alpha b^{rs}b^{mn}-b^{rm}b^{sn}\right)\nabla_1h_{mn}\nabla_1h_{rs}-2\alpha fK^{\alpha}\sum_{i}\sum_{\ell>\mu}\frac{(\nabla_ih_{1\ell})^2}{\lambda_1\lambda_i(\lambda_{\ell}-\lambda_1)}\nonumber\\
  &=&\alpha(\gamma^{-1}-1)\lambda_1^{-2}fK^{\alpha}\sum_i\frac{(h_{11,i})^2}{\lambda_i}-\alpha^2(1-\epsilon) fK^{\alpha}\lambda_1^{-1}\sum_{i,j}\frac{h_{ii,1}h_{jj,1}}{\lambda_i\lambda_j}
 \nonumber\\
  &&+\alpha fK^{\alpha}\lambda_1^{-1}\sum_{i,j}\frac{(h_{ij,1})^2}{\lambda_i\lambda_j}-2\alpha fK^{\alpha}\sum_{i}\sum_{\ell>\mu}\frac{(h_{1\ell,i})^2}{\lambda_1\lambda_i(\lambda_{\ell}-\lambda_1)}.
 \end{eqnarray}
By Lemma \ref{approx} and Codazzi equation, we get $\nabla_1h_{ij}=0$ for $1<j\leq\mu$, which implies
\begin{align*}
	\alpha fK^{\alpha}\lambda_1^{-1}\sum_{i,j}\frac{(h_{ij,1})^2}{\lambda_i\lambda_j}=&\alpha fK^{\alpha}\lambda_1^{-3}(h_{11,1})^2+2\alpha fK^{\alpha}\lambda_1^{-1}\sum_{j>\mu}\frac{(h_{11,j})^2}{\lambda_1\lambda_j}\nonumber\\
	&+\alpha fK^{\alpha}\lambda_1^{-1}\sum_{i,j>\mu}\frac{(h_{ij,1})^2}{\lambda_i\lambda_j}
\end{align*}
and
\begin{align*}
	2\alpha fK^{\alpha}\sum_{i}\sum_{\ell>\mu}\frac{(h_{1\ell,i})^2}{\lambda_1\lambda_i(\lambda_{\ell}-\lambda_1)}=&2\alpha fK^{\alpha}\sum_{\ell>\mu}\frac{(h_{11,\ell})^2}{\lambda_1^2(\lambda_{\ell}-\lambda_1)}+2\alpha fK^{\alpha}\sum_{k,\ell>\mu}\frac{(h_{k\ell,1})^2}{\lambda_1\lambda_k(\lambda_{\ell}-\lambda_1)}.
\end{align*}
Therefore,
\begin{align}\label{s4-eq6}
	&\alpha fK^{\alpha}\lambda_1^{-1}\sum_{i,j}\frac{(h_{ij,1})^2}{\lambda_i\lambda_j}-	2\alpha fK^{\alpha}\sum_{i}\sum_{\ell>\mu}\frac{(h_{1\ell,i})^2}{\lambda_1\lambda_i(\lambda_{\ell}-\lambda_1)}\nonumber\\
	\leq&\alpha fK^{\alpha}\lambda_1^{-3}(h_{11,1})^2-2\alpha fK^{\alpha}\sum_{\ell>\mu}\frac{(h_{11,\ell})^2}{\lambda_1\lambda_{\ell}(\lambda_{\ell}-\lambda_1)}.
\end{align}
Substituting \eqref{s4-eq6} into \eqref{s5.6}, and taking $\epsilon=\frac{\alpha}{1+\alpha}$ we have
 \begin{eqnarray*}
  Q_1  &\leq&\alpha(\gamma^{-1}-1)\lambda_1^{-2}fK^{\alpha}\sum_i\frac{(h_{11,i})^2}{\lambda_i}-\alpha^2(1-\epsilon) fK^{\alpha}\lambda_1^{-1}\sum_{i,j}\frac{h_{ii,1}h_{jj,1}}{\lambda_i\lambda_j}
  \nonumber\\
  &&+\alpha fK^{\alpha}\lambda_1^{-3}(h_{11,1})^2-2\alpha fK^{\alpha}\sum_{\ell>\mu}\frac{(h_{11,\ell})^2}{\lambda_1\lambda_{\ell}(\lambda_{\ell}-\lambda_1)}\\
  &\leq&\alpha(\gamma^{-1}-\alpha(1-\epsilon))fK^{\alpha}\lambda_1^{-3}(h_{11,1})^2\leq0.
 \end{eqnarray*}
 where we used that  $\gamma=n(1+\frac{1}{\alpha})=\frac{n}{\epsilon}>1$.

\textbf{Step 3. Estimate the $Q_0$ terms}. Return to \eqref{s5.5}, note that we have chosen $\epsilon=\frac{\alpha}{1+\alpha}$ and $\gamma=n(1+\alpha)/\alpha$ in \textbf{Step 2}.  We obtain that at the maximum point $(p_0,t_0)$
 \begin{align*}
  0\leq &n^2(1+\alpha) fK^{\alpha}\upsilon^{-1}\psi_{\beta}^{-1}-n\beta(1+\alpha)\alpha^{-1}\psi_{\beta}^{-1}+n\alpha fK^{\alpha}\lambda_1\nonumber\\
  & -\alpha fK^{\alpha}H+\frac{1+\alpha}{\alpha}f^{-1}\vert f_1\vert^2K^{\alpha}\lambda_1.
 \end{align*}
Equivalently,
 \begin{eqnarray*}
  n^2(1+\alpha) fK^{\alpha}\upsilon^{-1}\psi_{\beta}^{-1}
  \geq n\beta(1+\alpha)\alpha^{-1}\psi_{\beta}^{-1}+\alpha fK^{\alpha}H-\left(n\alpha +\frac{1+\alpha}{\alpha}f^{-2}\vert f_1\vert^2\right)fK^{\alpha}\lambda_{1}
 \end{eqnarray*}
 holds at  $(p_0,t_0)$. Multiplying $f^{-1}K^{-\alpha}$ and using $\upsilon>1$ yields
 \begin{eqnarray}\label{s4.q0-1}
   n^2(1+\alpha) \psi_{\beta}^{-1}\geq \alpha H+n\beta(1+\alpha)\alpha^{-1}\psi_{\beta}^{-1} f^{-1}K^{-\alpha}-\left(n\alpha +\frac{1+\alpha}{\alpha}f^{-2}\vert f_1\vert^2\right)\lambda_{1}.
 \end{eqnarray}
Observe that
\begin{equation*}
  H=\sum_{i=1}^n\lambda_i\geq \sum_{i=2}^n\lambda_i\geq (n-1)\left(\prod_{i=2}^n\lambda_i\right)^{\frac{1}{n-1}}=(n-1)K^{\frac{1}{n-1}}\lambda_{1}^{-\frac{1}{n-1}}.
\end{equation*}
We can estimate the first two terms on the right hand side of \eqref{s4.q0-1} using the Young's inequality:
 \begin{align*}
  \alpha H+n\beta(1+\alpha)\alpha^{-1}\psi_{\beta}^{-1} f^{-1}K^{-\alpha}\geq & \alpha(n-1)K^{\frac{1}{n-1}}\lambda_{1}^{-\frac{1}{n-1}}+{n\beta(1+\alpha)}{\alpha}^{-1} \psi_{\beta}^{-1}(\sup_{\mathbb{S}^n}f)^{-1}K^{-\alpha}\nonumber\\
  \geq &\frac{\alpha(n-1)}{1+(n-1)\alpha}K^{\frac{1}{n-1}}\lambda_{1}^{-\frac{1}{n-1}}+\frac{n\beta}{(1+(n-1)\alpha)} \psi_{\beta}^{-1}(\sup_{\mathbb{S}^n}f)^{-1}K^{-\alpha}\nonumber\\
  \geq &\left({n\beta}\psi_{\beta}^{-1}(\sup_{\mathbb{S}^n}f)^{-1}\right)^{\frac{1}{1+(n-1)\alpha}}\lambda_{1}^{-\frac{\alpha}{1+(n-1)\alpha}}.
 \end{align*}
So using $\psi_\beta\leq N$ we obtain from \eqref{s4.q0-1} that
  \begin{align*}  	
 \omega_0^{\frac{\alpha}{1+(n-1)\alpha}}=& \left(\psi_{\beta}^{n(1+\frac{1}{\alpha})}\lambda_{1}^{-1}\right)^{\frac{\alpha}{1+(n-1)\alpha}}\\
 \leq & \left(\frac{\sup_{\mathbb{S}^n}f}{n\beta}\right)^{\frac{1}{1+(n-1)\alpha}}\psi_{\beta}^{\frac{n+\alpha}{1+(n-1)\alpha}}\bigg(n^2(1+\alpha)\nonumber\\
  &\quad +\left(n\alpha+\frac{1+\alpha}{\alpha}\sup_{\mathbb{S}^n}|\bar{\nabla}\log f|^2\right)\psi_{\beta}\lambda_{1}\bigg)\\
  \leq &   \left(\frac{\sup_{\mathbb{S}^n}f}{n\beta}\right)^{\frac{1}{1+(n-1)\alpha}}N^{\frac{n+\alpha}{1+(n-1)\alpha}}\bigg(n^2(1+\alpha)\nonumber\\
  &\quad +\left(n\alpha+\frac{1+\alpha}{\alpha}\sup_{\mathbb{S}^n}|\bar{\nabla}\log f|^2\right)N^{1+n(1+\frac{1}{\alpha})}\psi_{\beta}^{-n(1+\frac{1}{\alpha})}\lambda_{1}\bigg),
 \end{align*}
which implies an upper bound of $\psi_{\beta}^{n(1+\frac{1}{\alpha})}\lambda_{\min}^{-1}$ by a constant depending on $n,\beta,N,\alpha$ and $f$. In fact, a precise estimate can be of the following form
\begin{align}\label{s4.w2}  	
\omega_0
  \leq &  \max\left\{ \left(\frac{\sup_{\mathbb{S}^n}f}{n\beta}\right)^{\frac{1}{\alpha}},1\right\}N^{\frac{n+\alpha}{\alpha}}\bigg(n^2(1+\alpha)\nonumber\\
  &\quad +\left(n\alpha+\frac{1+\alpha}{\alpha}\sup_{\mathbb{S}^n}|\bar{\nabla}\log f|^2\right)N^{n}\bigg)^{\frac{1+(n-1)\alpha}{\alpha}}\nonumber\\
  =:&C(N,f,n,\alpha, \beta)
 \end{align}
 for a constant $C(N,f,n,\alpha, \beta)$ depending on $N$, $f$, $n$,$\alpha$ and $\beta$. The estimate \eqref{s5.t1} follows from combining \eqref{s4.w1} and \eqref{s4.w2}.
  \end{proof}

 \section{Local speed estimate}\label{sec4}
 In this section, we estimate the local upper bound of the speed function $fK^{\alpha}$. Denote the cut-off function
\begin{equation}\label{s6.1}
\psi(p,t)=\psi_0(p,t)=(N-\bar{u}(p,t))_+,
\end{equation}
where $\bar{u}(p,t)=\langle F(p,t),e_{n+1}\rangle$ is the height function. We prove the following estimate.
 \begin{theorem}\label{speed estimate}
  Let $\Sigma_t$ be a complete strictly convex smooth graph solution of (\ref{flow1}) defined on $M^n\times[0,T]$, for some $T>0$. Then, given a
  constant  $N>0$,
  \begin{align}\label{s6.t1}
    \frac{t}{1+t}K^{\frac{1}{n}}\psi^{2}(p,t)\leq&~ (2\theta)^{1+\frac{1}{2n\alpha}} \max\{(\sup_{\mathbb{S}^n} f)^{\frac{1}{n\alpha}},1\} (\min_{\mathbb{S}^n}f)^{-\frac{1}{n\alpha}}\nonumber\\
  &\quad \times \bigg(N^{2}+2n\alpha \Big(N+\Lambda(4n\alpha+1+4n\alpha\theta)\Big)\bigg)
  \end{align}
for any $(p,t)\in M^n\times [0,T]$,  where  $\theta$ and $\Lambda$ are constants given by
  \begin{equation*}
    \theta=\sup\{\upsilon^2(p,s):\bar{u}(p,t)<N,s\in[0,t]\},
   \end{equation*}
   \begin{equation*}
    \Lambda=\sup\{\lambda^{-1}_{\min}(p,s):\bar{u}(p,t)<N,s\in[0,t]\}.
   \end{equation*}
 \end{theorem}
 \begin{proof}
To estimate the upper bound of $K$, we first look at the evolution equation \eqref{speed function} of $fK^{\alpha}$, which implies
 \begin{equation}\label{s5.fk}
  \partial_t(f^2K^{2\alpha})=\mathcal{L}(f^2K^{2\alpha})-\frac{1}{2}(f^2K^{2\alpha})^{-1}\Vert\nabla (f^2K^{2\alpha})\Vert^2_{\mathcal{L}}+2\alpha f^3K^{3\alpha}H.
 \end{equation}
 The reaction term on the right hand side of \eqref{s5.fk} is a bad term. We can improve this equation by multiplying a suitable factor $\varphi$, which we described below.  For any fixed time $T_0\in(0,T]$, define $\theta$ by
 \begin{equation*}
  \theta=\sup\{\upsilon^2(p,t):\bar{u}(p,t)<N,t\in[0,T_0]\}.
 \end{equation*}
We have $1\leq \upsilon^2\leq \theta$ on the support of $\psi$ for $t\in [0,T_0]$. Using an idea of Caffarelli, Nirenberg and Spruck in \cite{caffarelli1988nonlinear}
 (see also \cite{choi2019evolution} and \cite{ecker1991interior}), define the function $\varphi=\varphi(\upsilon^2)$ by
 \begin{equation}\label{test fun}
  \varphi(\upsilon^2)=\frac{\upsilon^2}{2\theta-\upsilon^2}.
 \end{equation}
Then $\varphi$ is well-defined on the support of $\psi$ and for $t\in [0,T_0]$ and satisfies  $\frac{1}{2\theta-1}\leq\varphi\leq 1$. Write for short
\begin{equation*}
  \varphi=\varphi(\upsilon^2),\quad \varphi'=\varphi'(\upsilon^2),\quad \mathrm{and}\quad \varphi''=\varphi''(\upsilon^2).
\end{equation*}
The evolution equation \eqref{gradient function} of $\upsilon$ leads to
 \begin{eqnarray}\label{s5.var}
  \partial_t\varphi&=&\varphi'\partial_t(\upsilon^2)=\varphi'2\upsilon
  \left(\mathcal{L} \upsilon-2\upsilon^{-1}\Vert \nabla\upsilon\Vert^2_{\mathcal{L}}-\alpha fK^{\alpha}H\upsilon\right)\nonumber\\
  &=&\varphi'\left(\mathcal{L} \upsilon^2-6\Vert\nabla\upsilon\Vert_{\mathcal{L}}^2-2\alpha fK^{\alpha}H\upsilon^2\right)\nonumber\\
  &=&\mathcal{L}\varphi(\upsilon^2)-\varphi''\Vert\nabla\upsilon^2\Vert_{\mathcal{L}}^2-
  \varphi'\left(6\Vert\nabla\upsilon\Vert_{\mathcal{L}}^2+2\alpha fK^{\alpha}H\upsilon^2\right).
 \end{eqnarray}
From \eqref{s5.fk} and \eqref{s5.var} we compute that
 \begin{eqnarray}\label{s5.fkvar}
  \partial_t(f^2K^{2\alpha}\varphi)&=&\mathcal{L}(f^2K^{2\alpha}\varphi)-2\langle\nabla\varphi,\nabla(f^2K^{2\alpha})\rangle_{\mathcal{L}}
  +2\alpha f^3K^{3\alpha}H(\varphi-\varphi'\upsilon^2)\nonumber\\
  &&-\frac{1}{2}\varphi(f^2K^{2\alpha})^{-1}\Vert\nabla (f^2K^{2\alpha})\Vert^2_{\mathcal{L}}-
  (4\varphi''\upsilon^2+6\varphi')f^2K^{2\alpha}\Vert\nabla\upsilon\Vert_{\mathcal{L}}^2.
 \end{eqnarray}
Note that $\varphi-\varphi'\upsilon^2=-\varphi^2$, so the third term on the right hand side of \eqref{s5.fkvar} introduces useful terms. We may estimate the second term on the right side of \eqref{s5.fkvar} modulo a gradient term by
\begin{align*}
  -2\langle\nabla\varphi,\nabla(f^2K^{2\alpha})\rangle_{\mathcal{L}}&+\varphi^{-1}\langle\nabla\varphi,\nabla(f^2K^{2\alpha}\varphi)\rangle_{\mathcal{L}}\nonumber\\
  =&-\langle\nabla\varphi,\nabla(f^2K^{2\alpha})\rangle_{\mathcal{L}}+\varphi^{-1}f^2K^{2\alpha}\Vert\nabla\varphi\Vert^2_{\mathcal{L}}\\
  \leq&\frac{1}{2}\varphi(f^2K^{2\alpha})^{-1}\Vert\nabla (f^2K^{2\alpha})\Vert^2_{\mathcal{L}}
  +\frac{3}{2}\varphi^{-1}f^2K^{2\alpha}\Vert\nabla\varphi\Vert^2_{\mathcal{L}}.
\end{align*}
Hence, the following inequality holds:
\begin{align*}
  \partial_t(f^2K^{2\alpha}\varphi)\leq&\mathcal{L}(f^2K^{2\alpha}\varphi)-\varphi^{-1}\langle\nabla\varphi,\nabla(f^2K^{2\alpha}\varphi)\rangle_{\mathcal{L}}
  -2\alpha f^3K^{3\alpha}H\varphi^2\\
  &-(4\varphi''\upsilon^2+6\varphi'-6\varphi^{-1}\varphi'^2\upsilon^2)f^2K^{2\alpha}\Vert\nabla\upsilon\Vert_{\mathcal{L}}^2.
 \end{align*}
To simplify the notation, we set
\begin{equation*}
  \chi:=f^2K^{2\alpha}\varphi.
\end{equation*}
 Observe that $\varphi$ satisfies
\begin{equation*}
  4\varphi''\upsilon^2+6\varphi'-6\varphi^{-1}\varphi'^2\upsilon^2=\frac{4\theta}{(2\theta-\upsilon^2)^2}\varphi.
\end{equation*}
Therefore we obtain the evolution equation
  \begin{equation}\label{s6.4}
  \partial_t\chi\leq\mathcal{L}\chi-4\theta\varphi\upsilon^{-3}\langle\nabla\upsilon,\nabla \chi\rangle_{\mathcal{L}}-2\alpha \chi fK^{\alpha}H\varphi
  -\frac{4\theta}{(2\theta-\upsilon^2)^2}\chi\Vert\nabla\upsilon\Vert^2_{\mathcal{L}}.
 \end{equation}

Next we localize the above calculation by multiplying the cut-off function $\psi(p,t)$ in space defined in \eqref{s6.1} and a cut-off function in time defined by
\begin{equation*}
  \eta=\frac{t}{1+t}.
\end{equation*}
We consider the following auxiliary function
\begin{equation}\label{s6.2}
 \omega=\eta^{2n\alpha}\psi^{4n\alpha}\chi.
\end{equation}
Then the maximum of $\omega$ is attained at some interior point $(p_0,t_0)$ for $p_0$ in the support of $\psi$ and $t_0>0$, allowing for $t_0=T_0$.

\begin{lemma}\label{s6.lem}
Along the flow \eqref{flow1}, on the support of $\psi$ and for $t>0$, we have
\begin{align}\label{s6.w1}
  \partial_t\ln\omega-\mathcal{L}\ln \omega\leq & ~ \frac{2n\alpha}{t(1+t)}+4n\alpha\psi^{-2}\Vert\nabla\psi\Vert^2_{\mathcal{L}}+\chi^{-2}\Vert\nabla \chi\Vert^2_{\mathcal{L}}+4n^2\alpha^2 fK^{\alpha}\psi^{-1}\nonumber\\
  &\quad-2\alpha fK^{\alpha}H\varphi -4\theta\varphi\upsilon^{-3}\chi^{-1}\langle\nabla\upsilon,\nabla \chi\rangle_{\mathcal{L}}-
  \frac{4\theta}{(2\theta-\upsilon^2)^2}\Vert\nabla\upsilon\Vert^2_{\mathcal{L}}.
\end{align}
\end{lemma}
\begin{proof}[Proof of Lemma \ref{s6.lem}]
By \eqref{cut-off fun1} and Proposition \ref{shift point}, the cut-off function $\psi$ satisfies
 \begin{eqnarray}\label{s6.5}
  \partial_t\psi(p,t)
    &=&\mathcal{L}\psi(p,t)+n\alpha fK^{\alpha}\upsilon^{-1}+
   K^{\alpha}\langle f(\nu)\nu+\nabla f(\nu), e_{n+1}\rangle\nonumber\\
   &\leq&\mathcal{L}\psi(p,t)+n\alpha fK^{\alpha}.
 \end{eqnarray}
 On the support of $\psi$ and for $t>0$, we combine \eqref{s6.4} - \eqref{s6.5} and obtain that
\begin{align*}
  \partial_t\ln(\eta^{2n\alpha}\psi^{4n\alpha}\chi)=& \frac{2n\alpha}{t(1+t)}+4n\alpha\psi^{-1}\partial_t\psi+\chi^{-1}\partial_t \chi\\
  \leq&\frac{2n\alpha}{t(1+t)}+4n\alpha\psi^{-1}\Big(\mathcal{L}\psi(p,t)+n\alpha fK^{\alpha}\Big) \\
  &+\chi^{-1}\Big(\mathcal{L}\chi-4\theta\varphi\upsilon^{-3}\langle\nabla\upsilon,\nabla \chi\rangle_{\mathcal{L}}-2\alpha \chi fK^{\alpha}H\varphi  -\frac{4\theta}{(2\theta-\upsilon^2)^2}\chi \Vert\nabla\upsilon\Vert^2_{\mathcal{L}}\Big)\\
  =&\mathcal{L}\ln(\eta^{2n\alpha}\psi^{4n\alpha}\chi)+\frac{2n\alpha}{t(1+t)}+4n\alpha\psi^{-2}\Vert\nabla\psi\Vert^2_{\mathcal{L}}+\chi^{-2}\Vert\nabla \chi\Vert^2_{\mathcal{L}}+4n^2\alpha^2 fK^{\alpha}\psi^{-1}\\
  &-4\theta\varphi\upsilon^{-3}\chi^{-1}\langle\nabla\upsilon,\nabla \chi\rangle_{\mathcal{L}}-2\alpha fK^{\alpha}H\varphi-
  \frac{4\theta}{(2\theta-\upsilon^2)^2}\Vert\nabla\upsilon\Vert^2_{\mathcal{L}}.
 \end{align*}
\end{proof}

Now, we continue the proof of Theorem \ref{speed estimate}. As $\omega$ attains its maximum in $M^n\times [0,T_0]$ at the point $(p_0,t_0)$, we clearly have $\omega(p_0,t_0)>0$ and so that $\ln\omega$ also attains its maximum at  $(p_0,t_0)$. This implies that
 \begin{eqnarray}\label{s5.dw}
  0=\nabla\ln\omega=4n\alpha\psi^{-1}\nabla\psi+\chi^{-1}\nabla \chi
 \end{eqnarray}
holds at  $(p_0,t_0)$. The the second and third terms on the right hand side of \eqref{s6.w1} satisfy
\begin{eqnarray*}
  4n\alpha\psi^{-2}\Vert\nabla\psi\Vert^2_{\mathcal{L}}+\chi^{-2}\Vert\nabla \chi\Vert^2_{\mathcal{L}}=4n\alpha(4n\alpha+1)\psi^{-2}\Vert\nabla\psi\Vert^2_{\mathcal{L}}
\end{eqnarray*}
at the point $(p_0,t_0)$. Applying \eqref{s5.dw} and the Cauchy-Schwarz inequality, we also estimate the second term of the second line of \eqref{s6.w1} at $(p_0,t_0)$
\begin{eqnarray*}
  -4\theta\varphi\upsilon^{-3}\chi^{-1}\langle\nabla\upsilon,\nabla \chi\rangle_{\mathcal{L}}&=&16\theta\varphi\upsilon^{-3}n\alpha\psi^{-1}\langle\nabla\upsilon,\nabla\psi\rangle_{\mathcal{L}}\\
  &\leq&\frac{4\theta}{(2\theta-\upsilon^2)^2}\Vert\nabla\upsilon\Vert^2_{\mathcal{L}}+16n^2\alpha^2\theta\upsilon^{-2}\psi^{-2}\Vert\nabla\psi\Vert^2_{\mathcal{L}}.
\end{eqnarray*}
Therefore, at $(p_0,t_0)$, we have
\begin{align}\label{s6.3}
  0\leq\partial_t\ln\omega-\mathcal{L}\ln\omega\leq&~\frac{2n\alpha}{t_0(1+t_0)}+
  4n^2\alpha^2fK^{\alpha}\psi^{-1}-2\alpha fK^{\alpha}H\varphi\nonumber\\
  &\quad +4n\alpha(4n\alpha+1+4n\alpha\theta\upsilon^{-2})\psi^{-2}\Vert\nabla\psi\Vert^2_{\mathcal{L}}.
\end{align}
Define
 \begin{equation*}
  \Lambda=\sup\{\lambda^{-1}_{\min}(p,t):\bar{u}(p,t)<N,t\in[0,T_0]\}.
 \end{equation*}
Then, on the support of $\psi$, we have
\begin{eqnarray}\label{s6.6}
  \Vert\nabla\psi\Vert^2_{\mathcal{L}}&=&\alpha fK^{\alpha}b^{ij}\langle \nabla_iF,e_{n+1}\rangle\langle \nabla_jF,e_{n+1}\rangle\nonumber\\
  &\leq&\alpha fK^{\alpha}b^{ij}g_{ij}\leq n\alpha fK^{\alpha}\lambda^{-1}_{\min}\leq n\alpha\Lambda fK^{\alpha}.
\end{eqnarray}
Using the facts $\upsilon\geq 1$, $H\geq nK^{\frac{1}{n}}$ and the estimate \eqref{s6.6}, we derive from \eqref{s6.3} that
\begin{align*}
  \psi^2K^{\frac{1}{n}}\varphi\leq & \frac{\psi^2}{t_0(1+t_0)fK^{\alpha}}+2n\alpha\bigg(\psi+\Lambda(4n\alpha+1+4n\alpha\theta)\bigg)
\end{align*}
holds at the point $(p_0,t_0)$. Then using $1\geq \varphi\geq\frac{1}{2\theta-1}$ and $\eta\leq 1$, we obtain
\begin{align*}
 \omega^{\frac{1}{2n\alpha}}(p_0,t_0)=&~\eta \psi^{2} f^{\frac{1}{n\alpha}}K^{\frac{1}{n}}\varphi^{\frac{1}{2n\alpha}}(p_0,t_0)\nonumber\\
 \leq &2\theta(\sup_{\mathbb{S}^n} f)^{\frac{1}{n\alpha}}\bigg(\frac{\psi^2}{(1+t_0)^2fK^{\alpha}}+2n\alpha\eta \Big(N+\Lambda(4n\alpha+1+4n\alpha\theta)\Big)\bigg)\\
 \leq &2\theta (\sup_{\mathbb{S}^n} f)^{\frac{1}{n\alpha}}\bigg(\psi^{2(n\alpha+1)}\left( \eta^{n\alpha} \psi^{2n\alpha}fK^{\alpha}\varphi^{1/2}\right)^{-1}
 +2n\alpha \Big(N+\Lambda(4n\alpha+1+4n\alpha\theta)\Big)\bigg)\nonumber\\
 \leq&2\theta (\sup_{\mathbb{S}^n} f)^{\frac{1}{n\alpha}}\bigg(N^{2(n\alpha+1)}\omega^{-1/2}
 +2n\alpha \Big(N+\Lambda(4n\alpha+1+4n\alpha\theta)\Big)\bigg).
\end{align*}
Whenever $\omega^{1/2}(p_0,t_0)\geq N^{2n\alpha}$, the last inequality yields
\begin{align*}
 \omega^{\frac{1}{2n\alpha}}(p_0,t_0)
 \leq &2\theta (\sup_{\mathbb{S}^n} f)^{\frac{1}{n\alpha}}\bigg(N^{2}+2n\alpha \Big(N+\Lambda(4n\alpha+1+4n\alpha\theta)\Big)\bigg).
\end{align*}
Since $\theta>1$, it follows that
\begin{align*}
 \omega^{\frac{1}{2n\alpha}}(p_0,t_0)
 \leq &2\theta \max\{(\sup_{\mathbb{S}^n} f)^{\frac{1}{n\alpha}},1\}\bigg(N^{2}+2n\alpha \Big(N+\Lambda(4n\alpha+1+4n\alpha\theta)\Big)\bigg).
\end{align*}
Using again $\varphi\geq\frac{1}{2\theta-1}$, we conclude that for any $(p,t)\in M^n\times[0,T_0]$,
\begin{align}\label{s6.speed2}
  \eta K^{\frac{1}{n}}\psi^{2}(p,t)=&~\omega^{\frac{1}{2n\alpha}}f^{-\frac{1}{n\alpha}}\varphi^{-\frac{1}{2n\alpha}}(p,t)\nonumber\\
  \leq&\omega^{\frac{1}{2n\alpha}}(p_0,t_0)(\min_{\mathbb{S}^n}f)^{-\frac{1}{n\alpha}}(2\theta-1)^{\frac{1}{2n\alpha}}\nonumber\\
  \leq&~(2\theta)^{1+\frac{1}{2n\alpha}}  \max\{(\sup_{\mathbb{S}^n} f)^{\frac{1}{n\alpha}},1\} (\min_{\mathbb{S}^n}f)^{-\frac{1}{n\alpha}}\nonumber\\
  &\quad \times \bigg(N^{2}+2n\alpha \Big(N+\Lambda(4n\alpha+1+4n\alpha\theta)\Big)\bigg).
\end{align}
In particular, \eqref{s6.speed2} holds for $t=T_0$.  As $T_0$ is arbitrary, \eqref{s6.t1} follows by setting $T_0$ as $t$.
\end{proof}

\section{Long-time existence}\label{sec5}
 In this section, we establish the long-time existence of the complete non-compact anisotropic $\alpha$-Gauss curvature flow \eqref{flow1} and complete the proof of  Theorem \ref{main thm}. The proof is based on the a priori estimates in \S\ref{sec2} - \S\ref{sec4} and it can be done by adapting the argument in \cite[\S 5]{choi2019evolution}, so we only sketch the proof.

To establish the long time existence of the smooth solution, the local estimates for higher order derivatives of the solutions are needed. We recall the following result which is a special case of Theorem 6 in \cite{And04}.
\begin{theorem}[\cite{And04}]\label{s6.thm1}
  Let $\Omega$ be a domain in $\mathbb{R}^n$. Let $u\in C^4(\Omega\times [0,T))$ be a function satisfying
  \begin{equation*}
    \frac{\partial}{\partial t}u=F(D^2u,Du,u,x,t),
  \end{equation*}
  where $F$ is $C^2$ and is elliptic, i.e., $\lambda I\leq (\dot{F}^{ij})\leq \Lambda I$ for some constants $\Lambda>\lambda>0$. Suppose that $F$ can be written as $F=\phi(G(D^2u,Du,u,x,t))$, where $G$ is concave with respect to $D^2u$ and $\phi$ is an increasing function on the range of $G$. Then in any relatively compact $\Omega'\subset\Omega$ and for any $\tau\in (0,T)$ we have
  \begin{equation*}
    \|u\|_{C^{2,\beta}(\Omega'\times (\tau,T))}\leq C,
  \end{equation*}
  where $\beta\in (0,1)$ depends on $n,\lambda$ and $\Lambda$, and $C$ depends on $\lambda,\Lambda,  \|u\|_{C^{2}(\Omega\times [0,T))},\tau, \mathrm{dist}(\Omega',\partial\Omega)$ and the bounds on the first and second derivatives of $G$.
\end{theorem}
The advantage of the above theorem is that it allows to relax the concavity hypothesis of the usual regularity theorem for fully nonlinear parabolic equation. In our case the equation \eqref{s2.flowu} is not concave with respect to $D^2u$ but is covered by Theorem \ref{s6.thm1}.  By the curvature lower bound in Theorem \ref{principle curvature} and the speed bound in Theorem \ref{speed estimate}, we have two-sides curvature bound from above and below. The estimate on $\upsilon$ in Theorem \ref{Gradient estimate} yields the $C^1$ estimate. Therefore, the ellipticity constants for the equation \eqref{s2.flowu} are controlled and the result in Theorem \ref{s6.thm1} applies to give the local $C^{2,\beta}$ estimate of the solution. The higher order derivatives estimate follows from the standard regularity theory.

 We recall some further notations.
  \begin{itemize}
            \item For a set $U\subset\mathbb{R}^{n+1}$, we denote the convex hull of $U$ by
      \[\text{Conv}(U)=\left\{tx+(1-t)y:x,y\in U, t\in[0,1]\right\}.\]
      \item Given a convex complete hypersurface $\Sigma$. If a set $V$ is a subset of $\text{Conv}(\Sigma)$, we say $V$ is enclosed by $\Sigma$ and use the notation $V\preceq\Sigma$. In particular, if $V\cap\Sigma=\emptyset$ and $V\preceq\Sigma$, we use  $V\prec\Sigma$.
      \item For a convex hypersurface $\Sigma$ and constant $\epsilon>0$, we use $\Sigma^{\epsilon}$ to denote its $\epsilon$-envelope
      \[\Sigma^{\epsilon}=\{Y\in\mathbb{R}^{n+1}:d(Y,\Sigma)=\epsilon, Y\notin\text{Conv}(\Sigma)\},\]
      where $d$ is the distance function in $\mathbb{R}^{n+1}$.
  \end{itemize}
 For the complete, non-compact locally uniformly convex hypersurface $\Sigma_0$ given in Theorem \ref{main thm}, Theorem \ref{cpt sup} implies that it is graph of a convex function $u_0$ over a domain $\Omega\subset\mathbb{R}^n$, that is $\Sigma_0=\{(x,u_0(x)):x\in\Omega\}$ for some function $u_0:\Omega\to \mathbb{R}$. The domain $\Omega$ may be bounded or unbounded. We first prove the existence of complete solution $\Sigma_t$ on $t\in(0,T)$, where $T$ depends on the given domain $\Omega$.

\begin{theorem}\label{longtime1}
  Let $\Sigma_0=\{(x,u_0(x)):x\in\Omega\}$ be a complete, non-compact locally uniformly convex hypersurface in $\mathbb{R}^{n+1}$. Assume that $B_{R}(x_0)\subset\Omega$ for some $R>0$, where $B_{R}(x_0)$ is an $n$-ball of radius $R$ centered at $x_0$.
  Then, given an immersion $F_0: M^n\to\mathbb{R}^{n+1}$ with $F_0(M^n)=\Sigma_0$, there
  is a solution $F: M^n\times(0,T)\to\mathbb{R}^{n+1}$ of (\ref{flow1}) for some $T\geq ((n\alpha+1)\sup_{\mathbb{S}^n} f)^{-1}R^{1+n\alpha}$ such that for
  each $t\in(0,T)$, the image $\Sigma_t:=F(M^n,t)$ is a strictly convex smooth complete graph of a function $u(\cdot,t):\Omega_t\to\mathbb{R}$
  defined on a convex open $\Omega_t\subset\Omega$.
\end{theorem}
\begin{proof}
  We will obtain a solution
  $\Sigma_t:=\{(x,u(x,t)):x\in\Omega_t\subset\mathbb{R}^n\}$ as a limit of an approximating sequence $\Sigma_t^j$. By Theorem \ref{cpt sup}, $\Sigma_0=\{(x,u_0(x)):x\in\Omega\}$ for some function $u_0:\Omega\to \mathbb{R}$. Assume that $\inf_{\Omega}u_0=0$. For each $j\in\mathbb{N}$, we reflect $\Sigma_0\cap(\mathbb{R}\times[0,j])$ over the $j$-level hyperplane to obtain a uniformly convex closed hypersurface $\overline{\Gamma}^j_0$, that is
  \[\overline{\Gamma}^j_0=\{(x,h)\in\mathbb{R}^{n+1}:h\in\{u_0(x),2j-u_0(x)\}, x\in\Omega, u_0(x)\leq j\}.\]
  Since $\overline{\Gamma}^j_0$ fails to be smooth, we approximate $\overline{\Gamma}^j_0$ by its ${1}/{j}$-envelope $\Gamma_0^j:=(\overline{\Gamma}^j_0)^{1/j}$, which is a uniformly convex closed hypersurface of class $C^{1,1}$. Then ,by \cite[Theorem 15]{andrews2000motion}, there is a unique closed and convex viscosity solution $\Gamma_t^j$ of \eqref{flow1} with initial data $\Gamma^j_0$ defined for $t\in(0,T_j)$, where $T_j$ is its maximal existing time. Note that the short time existence result in \cite[Theorem 15]{andrews2000motion} is proved for anisotropic Gauss curvature flow \eqref{flow}, it still holds for the modified flow \eqref{flow1} as they are equivalent up to a tangential diffeomorphism.  In addition, the symmetry of $\Gamma^j_t$ with the hyperplane $\mathbb{R}^n\times\{j\}$ can be obtained by the uniqueness of solution. Thus its lower half $\Sigma^j_t$
  is a graph of a function $u^j(\cdot,t)$ defined on a convex set $\Omega^j_t\subset\mathbb{R}^n$. Namely,
  \[\Sigma^j_t:=\Gamma_t^j\cap(\mathbb{R}^n\times[0,j])=\{(x,u^j(x,t)):x\in\Omega^j_t\}.\]
  We set
  \[\Sigma_t=\partial\{\cup_{j\in\mathbb{N}}\text{Conv}(\Gamma^j_t)\},\ \Omega_t=\cup_{j\in\mathbb{N}}\Omega^j_t,\ t\in[0,T),\]
  where $T=\sup_{j\in\mathbb{N}}T_j$.

  The speed of the flow \eqref{flow1} depends on the curvature and a function of the unit normal. So the comparison theorem still holds for the flow \eqref{flow1}. It follows that we can show $\Gamma^j_t\preceq\Gamma^{j+1}_t\preceq \Sigma_0$, which yields $u_0(y)\leq u^{j+1}(y,t)\leq u^j(y,t)$ and $T_j\leq T_{j+1}$.

  On the other hand, $B_{R}(x_0)\subset\Omega$ means that there is a constant $h_0$ satisfying $B_{R}^{n+1}(X_0)\preceq\Sigma_0$, where $B_{R}^{n+1}(X_0)$ is an $(n+1)$-ball of radius $R$ centered at the point $X_0=(x_0,h_0)$. Choose $j\geq R+h_0$ such that $B^{n+1}_R(X_0)\preceq\Gamma^j_0$ holds. Since $\partial B^{n+1}_{\rho_j}(X_0)$ is a super-solution of (\ref{flow1}), where
  \[\rho_j(t)=\left(R^{1+n\alpha}-(1+n\alpha)t\sup_{\mathbb{S}^n} f\right)^{\frac{1}{1+n\alpha}}.\]
  Hence, the comparison principle leads to $\partial B^{n+1}_{\rho_j(t)}(X_0)\preceq\Gamma^j_t$. Thus,
  $$T\geq T_j\geq((n\alpha+1)\sup_{\mathbb{S}^n} f)^{-1}R^{1+n\alpha}.$$

  Next, by the same manner as in the proof in \cite[Theorem 5.1]{choi2019evolution}, we conclude that the a priori estimates in \S \ref{sec2} - \S \ref{sec4} hold for $\Gamma^j_t$ and this implies that $\Sigma_t^j$ is a smooth strictly convex graph of function $u^j$. Finally, we can pass $\Gamma^j_t$ and $u^j$ to the limit  to get the desired result.
\end{proof}

To finish the proof of Theorem \ref{main thm}, we need to show that  each $\Sigma_t$ remains as a graph over the same domain $\Omega$ for all $t\in(0,T)$, which implies $T=+\infty$ independently from the domain $\Omega$. This can be proved by constructing an appropriate barrier as in \cite[Theorem 5.4]{choi2019evolution}).
\begin{theorem}\label{longtime2}
  Let $\Omega$, $u_0$ and $\Sigma_0$ satisfy the condition in Theorem \ref{longtime1}. Assume that $\Sigma_t=\{(x,u(x,t)):\Omega_t\}$, $t\in(0,T)$, is the
  strictly convex smooth complete graph solution of \eqref{flow1} such that $\Omega_t$, $u(\cdot,t)$ and $\Sigma_t$ satisfy the conditions of
  $\Omega$, $u_0$ and $\Sigma_0$ in Theorem \ref{longtime1}. Then, for any closed ball $\overline{B_{R_0}(y_0)}\subset\Omega$ and any $t_0\in(0,T)$, there
  holds $B_{R_0}(y_0)\subset\Omega_{t_0}$.
\end{theorem}
\begin{proof}
 The proof is similar as in \cite[Theorem 5.4]{choi2019evolution}) with some minor modification including the change of constants involving the anisotropy function $f$ and the definition of supersolution of the modified flow \eqref{flow1}. We omit the details here and refer the readers to \cite[Theorem 5.4]{choi2019evolution}).
 \end{proof}

\bibliographystyle{plain} 

\end{document}